\documentclass[12pt,reqno,epsfig,psfig]{article}
\usepackage{amsmath}
\usepackage{amssymb}
\usepackage{amsfonts}
\usepackage{amsthm}
\usepackage{enumerate}
\usepackage[latin1]{inputenc}
\usepackage[T1]{fontenc}
\usepackage{float}

\newcommand{\comment}[1]{}

\newcommand{\al}{{\alpha}}
\newcommand{\ve}{{\varepsilon}}
\newcommand{\si}{\sigma}

\newcommand{\la}{\lambda}

\newcommand{\Om}{\Omega}
\newcommand{\om}{\omega}

\newcommand{\comp}{\mathbb C}

\newcommand{\reel}{\mathbb R}

\newcommand{\CC}{\mathbb C}
\newcommand{\NN}{\mathbb N}
\newcommand{\RR}{\mathbb R}

\newcommand{\ZZ}{\mathbb Z}

\newcommand{\AAA}{\mathcal{A}}

\newtheorem{theoreme}{Theorem}[section]
\newtheorem{definition}{Definition}[section]
\newtheorem{proposition}{Proposition}[section]

\newtheorem{lemme}{Lemma}[section]
\newtheorem{lemma}{Lemma}[section]

\newtheorem{corollary}{Corollary}[section]

\newtheorem{remark}{Remark}[section]

\newtheorem*{notation*}{Notations}
\theoremstyle{remark}

\newcounter{rev}
\begin{document}

\title{Generalization of the effective Wiener-Ikehara theorem}

\author{Anne de Roton \& Szil\' ard Gy. R\' ev\' esz{\thanks{Supported in part by the Hungarian National Foundation for Scientific Research, Project \#s K-81658 and K-100461. Work done in the framework of the project ERC-AdG 228005.}}}

\let\oldfootnote\thefootnote
\def\thefootnote{}

\maketitle

\centerline{\emph{Dedicated to J\'anos Pintz on the occasion of his sixtieth birthday}}

{\small \tableofcontents}

\begin{abstract} We consider the classical Wiener-Ikehara Tauberian theorem, with a generalized condition of slow decrease and some additional poles on the boundary of convergence of the Laplace transform. In this generality, we prove the otherwise known asymptotic evaluation of the transformed function, when the usual conditions of the Wiener-Ikehara theorem hold. However, our version also provides an effective error term, not known thus far in this generalality. The crux of the proof is a proper, asymptotic variation of the lemmas of Ganelius and Tenenbaum, also constructed for the sake of an effective version of the Wiener-Ikehara Theorem.
\end{abstract}



\baselineskip=15pt

\maketitle

%%%% \newpage

\section{Introduction}

For a carefully written, detailed account of the history and extensive work in the area of the Tauberian theorems of the family usually labeled as Wiener-Ikehara Theorem, we emphasize \cite[p.125]{Ko}. Our point here is the discussion of general versions of these results.

We first recall Theorem 4.2 of \cite{Ko}, the integral form of the classical Wiener-Ikehara Tauberian theorem.
\begin{theoreme}[Wiener-Ikehara]\label{class_Wienik}
Let $A(t)$ vanish on $(-\infty,0)$, be nondecreasing, continuous from the right and such that the Laplace-Stieltjes transform
$\mathcal{A}(s)=\int_0^\infty e^{-st}dA(t)$ is convergent for $\Re{s}>1$.
Suppose that for some constant $c$, the analytic function
\begin{equation}\label{eq:Gdef}
G(s):=\frac{1}{s+1}\mathcal{A}(s+1)-\frac{c}{s}
\end{equation}
has a boundary function in the following sense. For $x$ tending to $1$ from the right, the function $G_x(iy)=G(x+iy)$ converges to $G(1+iy)$ either uniformly or in $L^1$ on every finite interval $-T<y<T$. Then
$$\lim_{t\rightarrow +\infty}e^{-t}A(t)= c.$$
\end{theoreme}

\begin{remark}\label{l:convergenceetordre}
The convergence of the integral $\mathcal{A}(s)$ for $\Re{s}>1$ implies in itself an upper bound for the summatory function $A(t)$ of  locally bounded variation on $\reel^+$ ($A$ does not even need to be nondecreasing). Integration by parts indeed leads to the equivalence between the convergence of the Laplace-Stieltjes transform for $\Re{s}>1$ and the estimation $|A(t)|\ll_{\ve} e^{(1+\varepsilon)t} $for any positive $\ve$.
\end{remark}

Theorem \ref{class_Wienik} has been generalised in different ways.
\begin{itemize}
\item The Tauberian condition "$A$ nondecreasing" can be relaxed to some boundedly or slowly decreasing conditions, see \cite[p.135, 141, 142]{Ko} but in case of bounded decrease, one has to strengthen the regularity of the Laplace-Stieltjes transform on the boundary by assuming its analyticity on the border line.
\item One may allow $\mathcal{A}(s)$ to have some more general singularity at $s=1$ \cite[p.326]{T2} and one may take $T$ in Theorem \ref{class_Wienik} to be fixed or even as small as we want but this may lead to a boundedness result rather than a convergence result \cite[p.128, 142]{Ko}.
\item There also exists some effective Wiener-Ikehara theorem giving an explicit error term \cite[p.326]{T2}. In order to get an explicit error term, we need to know more about the behavior of the Laplace-Stieltjes transform $\mathcal{A}(s)$ near and on the border line ($T$ has to be taken large) and we cannot relax too much the Tauberian conditions.
\end{itemize}

In this work, we give some boundedness theorem for slowly and even boundedly decreasing functions $A$ under some assumption of local behavior of $\mathcal{A}(s)$ and an effective Wiener-Ikehara theorem for functions $A$ admitting a slowly or moderately decreasing condition (to be defined later) under some regularity assumptions of $\mathcal{A}(s)$ on the border line.

Section \ref{sec:slowdecrease} will be devoted to the definitions of boundedly, slowly and moderately decreasing functions and to the links between these notions. We recall and state some preliminary results in Section \ref{preliminary}. In Section \ref{sec:boundedness}, we present the boundedness Tauberian theorem under local assumptions. Finally, in Section \ref{sec:multiplepoles}, we shall present the generalized effective Wiener-Ikehara theorems. The proof of these results will require some lemma close to the Ganelius lemma as generalized by Tenenbaum in \cite[p.328]{T2}. This lemma will be given in Section \ref{sec:Ganelius}.

\bigskip

We shall use in the following the notion of functions of bounded variation, see e.g. \cite[pp. 116--]{SS}. For a short account of the notion explained in the Lebesgue integral context, i.e. for Lebesgue-Stieltjes integrals, see e.g. Appendix 6 on p. 436-437 of \cite{BGT}.
\begin{notation*}
If $I$ is an interval and $f$ a function defined on an interval containing $I$, we shall write $\|f\|_{I}$ for $\sup_{x\in I}|f(x)|$.

We define the Fourier transform of a function $f\in L^1(\RR)$ as
\begin{equation*}
    \widehat{f}(\tau)=\int_{-\infty}^{+\infty}e^{-i\tau x}f(x)dx \quad (\tau\in\RR).
\end{equation*}
\end{notation*}
%%%%%%%%%%%%%%%%%%%%%%%%%%%%%%%%%%%%%%%%%%%%%%%%%%%%%%%%%%%%%%%%%%%

\section{Conditions on controlled decrease of functions}\label{sec:slowdecrease}

Here we recall and introduce some relevant function theoretic notions
-- that of \emph{slowly decreasing functions}, introduced by Schmidt
in \cite{S1, S2}, an essentially weaker version of the property,
called here \emph{(locally uniformly) bounded decrease} and used in a
similar form e.g. by Korevaar \cite[Proposition 10.2. (i)]{Ko} and a
variation of the notion, which we will term as \emph{moderate
  decrease}. Our analysis will cover some aspects not
found in the literature and altogether is meant %%%%aims at serving as a
to be a self-contained precursor to the later applications of these properties
as Tauberian conditions in Sections \ref{sec:boundedness} and
\ref{sec:multiplepoles}.

Let us start with a general analysis of real functions $f:[a,\infty) \to \RR$, where $a>0$ is an arbitrarily fixed parameter. We can always consider, for arbitrary $\lambda \geq 1$, the quantities
\begin{equation}\label{eq:nubardef}
\overline{\nu}(f;\lambda):=\overline{\nu}(\lambda):=-\inf \{ f(y)-f(x)~ : ~ a \le x \le y \le \lambda x\}
\end{equation}
and
\begin{align}\label{eq:nudef}
\nu(f;\lambda):=\nu(\lambda)&:=-\liminf_{ x \to \infty} \{
f(y)-f(x)~ : ~ a \le x \le y \le \lambda x \}
 \\ \notag & =-\lim_{ z \to \infty} \Big(\inf \{ f(y)-f(x)~ : ~ z \le x \le y \le \lambda x \}\Big).
\end{align}
The definition of $\nu(\lambda)$ and $\overline{\nu}(\lambda)$ entails that for any $\lambda\geq 1$ there exists a nonincreasing function $\Psi_{\lambda}$ with $\lim\limits_{x\to\infty} \Psi_\lambda(x)=0$ satisfying
\begin{equation}\label{eq:nuPsilambda}
f(y)-f(x) \geq -\nu(\lambda)-\Psi_{\lambda}(x) \geq - \overline{\nu}(\lambda) \quad \textrm{whenever}\quad a\leq x\leq y \leq \lambda x.
\end{equation}
Clearly $\nu(1)=\overline{\nu}(1)=0$ and $0\leq {\nu}(\lambda) \leq
\overline{\nu}(\lambda)$. Both quantities are nondecreasing functions of their variable $\lambda\in [1,\infty)$ and are subadditive in the sense that for any values $ \lambda_1,\lambda_2\geq 1$,
we have $ {\nu}(\lambda_1\lambda_2)\leq
{\nu}(\lambda_1)+\nu(\lambda_2)$ and similarly for
$\overline{\nu}$. Of course the above remains valid even if
the values become infinite, so when $\nu$ or $\overline{\nu}$ are
functions only in the extended sense. However, once for \emph{any}
value $\lambda_0>1$ we have $\overline{\nu}(\lambda_0)$ or
$\nu(\lambda_0)$ finite, then by monotonicity and subadditivity
\emph{all} values remain finite. As we will always consider locally bounded functions $f$,
in fact for us finiteness of $\overline{\nu}$ and $\nu$ will be equivalent,
although in general there exist functions
which are unbounded and oscillate badly at some point and thus have
$\nu(\lambda)<+\infty$ whereas $ \overline{\nu}(\lambda)=+\infty$.

The first notion which we will use further as a condition on the decrease of a function is thus expressed by the finiteness of these characteristics.

\begin{definition}\label{def:Korevaar}
We say that a real function $f$ defined on $[a,+\infty)$ is \emph{boundedly decreasing}
if we have $\overline{\nu}(f;\lambda)<\infty$ for some -- and hence for all -- $\lambda >1$.
%%%%%\begin{equation}\label{eq:Korevaartailored}
%%%%%:=-\inf\{ f(y)-f(x)~ : ~ 1\le x \le y \le \lambda x\}< \infty
%%%%%\end{equation}
\end{definition}
Applying Definition \ref{def:Korevaar} on intervals of the form $[e^kx,e^{k+1}x]$ with $\lambda=e$, we infer for any $\lambda\geq e$ the lower bound
\begin{equation}\label{minlambda}
f(y) - f(x) \geq - \overline{\nu}(f;e)\lceil \log\lambda \rceil \qquad \textrm{whenever}\qquad (a \leq) x\leq y \leq \lambda x.
\end{equation}
In this work we basically consider functions $F$ with $F(x)/x$ boundedly decreasing on $[a,\infty)$ for some $a>0$. Taking $x=a$ in inequality \eqref{minlambda}, we see that if $F(x)/x$ is a boundedly decreasing function, then for $y\geq a$
\begin{equation}\label{minslowdec}
F(y) \geq - M y(1+\log{y}-\log{a}) \geq - 2 M y\max(1-\log{a},\log{y})
\end{equation}
with $M:=M(F,a):=\overline{\nu}(F(x)/x;e)+F(a)_{-}/a$.

Subadditivity, nonnegativity and monotonicity entails that whenever $\overline{\nu}$ is continuous at 1, that is when we have $\lim_{\lambda\to 1} \overline{\nu}(\lambda) =\overline{\nu}(1)=0$, then $\overline{\nu}$ is continuous everywhere on $[1,\infty)$. (It is clear that this limit must be at least 0, the question is if it is indeed zero.) The situation is the same for $\nu$.
\begin{definition}[Schmidt]\label{def:Schmidt}
Let $a>0$. We say that a real function $f$ defined on $[a,+\infty)$ is \emph{slowly decreasing} if we have
\begin{equation}\label{eq:Schmidttailored}
\lim\limits_{\lambda\downarrow 1} \liminf\limits_{x\to\infty} \inf\limits_{x\leq y\leq\lambda x} \left\{f(y) - f(x) \right\}= 0,
\end{equation}
that is if $~\lim_{\lambda\downarrow 1} \nu(\lambda) =0$, and consequently $\nu$ is a nonnegative, finite, continuous, subadditive, nondecreasing function of $\lambda\in [1,\infty)$.
\end{definition}

Ever since R. Schmidt's pioneering work \cite{S1, S2}, slowly decreasing functions have been used extensively in the context of Tauberian theorems, see e.g. \cite[Ch. III, \S 8]{Wi}, \cite[p.41]{BGT}, \cite[p.33, p.69]{Ko}.
Properties of slowly decreasing functions were already well-explored by Schmidt \cite{S1}, \cite{S2}, and can be found in several places, see e.g. \cite[p. 298]{Wi}.
%Nevertheless, to keep this article self contained, we shall rather present here the basic properties we need for such functions.

In the following we will encounter various restrictions on the decrease of a function as Tauberian conditions in deriving an effective Wiener-Ikehara Theorem. A usual condition in this context, (and one used also in Karamata's Theorem) is that $F(x)/(x\log^m x)$ is slowly decreasing. One would think at first sight, that once $F(x)/x$ is slowly decreasing, so does the further divided $F(x)/(x\log^m x)$, but caution is aroused in seeing that division by a nice increasing function in itself does not preserve slow decrease, as is the case with division by $x$, see Corollary \ref{cor:dividebylog}.

The next proposition will give, for some nondecreasing function $\ell$,  the order estimates on a general slowly or boundedly decreasing function $q(x)$ that guarantee that also $q/\ell$ be slowly or boundedly decreasing.
\begin{proposition}\label{prop:dividebyfi}
Let $a>0$ be a real number and $\ell:[a,+\infty)\rightarrow (0,\infty)$ be a positive, nondecreasing, differentiable function.
If $\ell'/\ell$ is nonincreasing, then for all slowly (respectively boundedly) decreasing functions $q(x)$ satisfying $q(x)=O(\ell^2(x)/(x\ell'(x)))$ we have $q/\ell$ also slowly (respectively boundedly) decreasing on $[a,+\infty)$.
\end{proposition}
\begin{remark}
Analogous results can be derived for more general functions $\ell$. Since all the functions $\ell$ which we shall use satisfy the assumptions in the proposition, we do not state the proposition in its full generality.
\end{remark}

\begin{corollary}\label{cor:dividebylog} Let $q$ be a slowly (respectively boundedly) decreasing function on $[1,+\infty)$.
If $m$ is a positive real number such that $q(x)=O(\log^{m+1} x)$ for $x\geq e$, then also $q(x)/\log^m x$ is slowly (respectively boundedly) decreasing on $[1,+\infty)$ whereas if $q(x)=O(x)$ for $x\geq 1$, then also $q(x)/x$ is slowly (respectively boundedly) decreasing on $[1,+\infty)$.

Then again, for any positive function $\omega$ tending to infinity, however slowly, there exists a slowly decreasing function $q$ on $[1,+\infty)$ such that $q(x)=O(x\omega(x))$ for $x$ large enough, but $q(x)/x$ is not even boundedly decreasing.
\end{corollary}
Note that the first statement of the Corollary shows that for certain functions $\ell$, $q(x)=O(\ell(x))$ is not the ultimate (most general) order condition on a slowly decreasing function $q$ to ensure slow decrease of $q /\ell$; but the last statement exemplifies that for $\ell(x)=x$ it is.
\begin{proof}[Proof of the Corollary] Proposition \ref{prop:dividebyfi} with the choices $\ell(x)=\log^m x$ and $\ell(x)=x$, respectively, yields the first two assertions.

To see the last, we first define a nondecreasing function $s_0$ which tends to infinity as $x$ tends to infinity by $s_0(x):=\inf_{x'\geq x} \omega(x')$. Then we define the function $s$ on$[1,\infty)$ by $s(1)=0$ and for any integer $n\geq 0$,
$$
s(x)=\min(s_0(x),s(e^n)+1), \quad \left(x\in(e^n,e^{n+1}]\right).
$$
The function $s$ is nondecreasing and unbounded. Furthermore it satisfies $0\leq s\leq s_0\leq \omega$, and $s(ex)\leq s(x)+2$ for $x\geq 1$.
The function $q$ defined by $q(x):=e^{\lfloor\log x\rfloor}s(x)$ satisfies $q(x)=O(x\omega(x))$ and is nondecreasing. Nevertheless, $q(x)/x$ is not slowly decreasing since for arbitrary $1<\lambda<e$,
$$
\inf_{n\in\NN}\frac{q(\lambda e^n)}{\lambda e^n}-\frac{q(e^n)}{e^n} =\inf_{n}\frac{e^{n}s(\lambda e^n)}{\lambda e^n}-\frac{e^ns(e^n)}{e^n}
\leq\inf_{n}\frac{2+(1-\lambda)s(e^n)}{\lambda}=-\infty.
$$
\end{proof}
\begin{proof}[Proof of Proposition \ref{prop:dividebyfi}]
Let $C$ be a constant satisfying $q(x)\leq C\dfrac{\ell^2(x)}{x\ell'(x)}$ for any $x\geq a$.

We write for arbitrary $a\leq x\leq y\leq \lambda x$, $\lambda>1$
\begin{align}\label{eq:qperfidifference}
\frac{q(y)}{\ell(y)}- \frac{q(x)}{\ell(x)} &= \frac{q(y)-q(x)}{\ell(y)}-q(x) \left( \frac{1}{\ell(x)} - \frac{1}{\ell(y)} \right).
\end{align}
By the Mean Value Theorem, there exists $\xi\in(x,y)$ such that
\begin{align*}
\left|q(x) \left( \frac{1}{\ell(x)} - \frac{1}{\ell(y)} \right) \right|&=\frac{|q(x)|}{\ell(x)} \frac{\ell(y)-\ell(x)}{\ell(y)}  =  \frac{|q(x)|}{\ell(x)} \frac{\ell'(\xi)(y-x)}{\ell(y)}\\
& \leq \frac{|q(x)|}{\ell(x)} \frac{\ell'(\xi)}{\ell(\xi)} (\lambda-1)x
\leq (\lambda-1) \frac{|q(x)| x\ell'(x)}{\ell^2(x)} \\& \leq C (\lambda-1),
\end{align*}
in view of the monotonicity of $\ell$, $\ell'/\ell$ and the growth order assumption on $q$.
Writing this last formula in \eqref{eq:qperfidifference} yields
$$\frac{q(y)}{\ell(y)}- \frac{q(x)}{\ell(x)} \geq  -\frac{\{q(y)-q(x)\}_-}{\ell(a)}-C (\lambda-1).$$
This gives
$$\frac{q(y)}{\ell(y)}- \frac{q(x)}{\ell(x)} \geq -\frac{\overline{\nu}(\lambda)}{\ell(a)}-C (\lambda-1)$$
in case $q$ is boundedly decreasing and
$$
\liminf_{x\rightarrow +\infty}\inf_{x\leq y\leq\lambda x} \left(\frac{q(y)}{\ell(y)}- \frac{q(x)}{\ell(x)} \right) \geq  - \frac{1}{\ell(a)}\nu(\lambda) - C(\lambda-1),
$$
which tends to 0 when $\lambda\to 1$ in case  $q$ is slowly decreasing . That concludes the proof.
\end{proof}

We introduce now a new class, that of \emph{moderately decreasing} functions, which will have a natural occurrence in our analysis. Some similar, but less general, condition for functions were considered by Korevaar in \cite{Ko}, p.382.

\begin{definition}\label{def:moderatedec}
Let $a\geq0$ be arbitrary. A real function $F$ on $[a,+\infty)$ is \emph{moderately decreasing} if there exist some positive constants $B_1, B_2$ such that for any pair $(u,v)$ of real numbers satisfying $v\geq 0$ and $u\geq a$ we have
\begin{equation}\label{eq:moderatedec}
F(u+v)-F(u) \geq -B_1 v - B_2 \max(1,u)\varphi(u)
\end{equation}
where $\varphi$ is a nonincreasing function on $[a,+\infty)$ satisfying $\varphi(x)=1$ for any $x\in[a,\max(1,a)]$ and $\lim_{u\rightarrow\infty}\varphi(u)=0$.
\end{definition}

From the definition applied to $u=a$ and $v=x-a\geq 0$ one obtains immediately
that for $x\geq \max(a,1)$
\begin{equation}\label{minF}
F(x)\geq - B_1(x-a) -B_2 \max(a,1) - F(a)_{-} \geq  -B x
\end{equation}
with $B:=B(F,a):=B(F,a,B_1,B_2):=B_1 + F(a)_{-}+B_2 $.

Note that \emph{upper} bounds for $F(x)$ \emph{cannot} be derived from the moderately decreasing property, since any increasing function is necessarily moderately decreasing.

In Theorem \ref{th:kleTfixe} we will show, however, that in case we have some control on the Mellin-Stieltjes transform of a moderately decreasing function $A$, then it entails some respective non-trivial upper bound on $A$.

\begin{remark}\label{rem:slowandmoderatenonequiv}
As proved in the proof of Corollary \ref{cor:dividebylog}, there are some increasing (hence moderately decreasing) functions $F$ such that $F(x)/x$ is not even boundedly decreasing, hence not slowly decreasing either.

Conversely, there are some functions $F$ which are not moderately decreasing whereas $F(x)/x$ is slowly decreasing. Even more so, $F(x)/x$ can be chosen \emph{very slowly decreasing} as introduced e.g. in \cite[formula (10.2), page 143]{Ko} and meaning that $\nu(\lambda)=0$ for some -- and hence (in view of monotonicity and subadditivity) for all -- $\lambda>1$.

For instance if $F(x)=- x\log x/\log\log x$, then $F(x)/x$ is very slowly decreasing, but for $u$ large and $v=u/\sqrt{\log u}$, condition \eqref{eq:moderatedec} fails for $F$, moreover, not even condition \eqref{minF} holds true for $x$ large.
\end{remark}

Nevertheless, for well-bounded functions there is \emph{some} connection between these notions of controlled decrease.

\begin{proposition}\label{prop:modOxslow}
Let $a$ and $A$ be real numbers satisfying $1\leq a\leq A$ and $A\geq e$, $m>0$ and $\ell(x):=\log^m x$. If a real function $F$ defined and moderately decreasing on $[a,\infty)$ satisfies $F(x)=O(x\ell(x))$ for $x\geq A$, then $F(x)/(x\ell(x))$ is slowly decreasing on $[A,\infty)$. However, the converse fails to hold, as there exist functions $F(x)=O(x)$ with $F(x)/x$ slowly decreasing, but $F(x)$ not moderately decreasing.
\end{proposition}
\begin{remark}
Corollary \ref{cor:dividebylog} gives that if $F(x)=O(x)$ with $F(x)/x$ slowly decreasing, then $F(x)/(x\ell(x))$ is also slowly decreasing. Therefore the second assertion really proves that the converse of the first one does not hold.
\end{remark}

\begin{proof}[Proof of Proposition \ref{prop:modOxslow}]
Assume that $F$ satisfies \eqref{eq:moderatedec} and $|F(x)|\leq Cx\ell(x)$ for $x\geq A$ and some positive constant $C$. Consider for any $1<\lambda<\log 2$ and $A\leq x\leq y\leq \lambda x$ the estimation
\begin{align*}
\frac{F(y)}{y\ell(y)} - \frac{F(x)}{x\ell(x)} & = \frac{F(y)-F(x)}{y\ell(y)} - \frac{y\ell(y)-x\ell(x)}{y\ell(y)}\frac{F(x)}{x\ell(x)} \\ &\geq - \frac{B_1(y-x) +B_2\varphi(x) x}{y\ell(y)} - C\left( \frac{y-x}{y}+\frac{x}{y}\left(1-\frac{\ell(x)}{\ell(y)}\right)\right)
\\ & \geq - B_2\varphi(x)-(B_1+C)\frac{\lambda-1}{\lambda}-C\left(1-\frac{\ell(x)}{\ell(\lambda x)}\right) ,
\end{align*}
which is already independent of $y\in[x,\lambda x]$. Taking $\lim\limits_{\lambda\downarrow 1} \liminf\limits_{x\to\infty}$ gives that $F(x)/(x\ell(x))$ is slowly decreasing.

To disprove a converse implication, first we define a slowly decreasing and bounded function $f(x)$ with $\nu(\lambda)$ to be large in a sense to be made more precise later.

Let $(m_n)$ be a sequence of integers at least $2$ covering all the integers of $[2,+\infty)$ infinitely often. We define the function $f$ on $[1,+\infty)$ as follows:
$$f(x)=\left\{\begin{array}{lll}
-\frac{\sqrt{m_n}}{2^n}(x-2^n) & \mbox{ if } &x\in I_n:=[2^n,\frac{m_n+1}{m_n}2^n)\\
-\frac{1}{\sqrt{m_n}} & \mbox{ if } &x\in J_n:=[\frac{m_n+1}{m_n}2^n,2^{n+1})
\end{array} ~.
\right.$$
Clearly  $\frac{-1}{\sqrt{2}}\leq \frac{-1}{\sqrt{m_n}} \leq f\leq 0$ on any interval $[2^n,2^{n+1})$, so $f(x)=O(1)$.

Next we show that $f$ is indeed slowly decreasing. For any $x\in[1,+\infty)$ and any $\lambda\in(1,4/3]$, the interval $[x,\lambda x]$ can mesh with at most one interval $I_n$, in which case we have  for any $y\in[x,\lambda x]$,
\begin{align*}
f(y)-f(x) &\geq -\frac{\sqrt{m_n}}{2^n}\min\left(y-x,|I_n| \right) \geq -\frac{\sqrt{m_n}}{2^{n}} \min\left( (\lambda-1) 2^{n+1}, \frac{2^n }{m_n} \right) \\
&\geq - \min\left(2(\lambda-1)\sqrt{m_n},\frac{1}{\sqrt{m_n}} \right) \geq - \sqrt{2(\lambda-1)}
\end{align*}
since $2^{n-1}<x<2^{n+1}$.
In case $[x,\lambda x]$ does not mesh with any interval $I_n$, this is also satisfied since $f(y)-f(x) =0$ for any $y \in[x,\lambda x]$.
This finally proves that $\nu(\lambda)\leq \sqrt{2(\lambda-1)}$ and thus $f$ is slowly decreasing.

So now we define $F(x)=xf(x)$. Then $F(x)=O(x)$, $F(x)/x$ is slowly decreasing, and it remains to see that nevertheless $F(x)$ is \emph{not} moderately decreasing. To do this, we show that for any given fixed positive constant $C$, there exists some $c>0$ such that the difference $F(x+y)-F(x)+Cy$ stays below $-c x$ for some choice of $x,y$ tending to $\infty$.

Making use of the fact that $F(x)\leq 0$, we write
\begin{align*}
F(x+y) - F(x) &\leq  (x+y) \left( \frac{F(x+y)}{x+y} -\frac{F(x)}{x}  \right) = (x+y) \left( f(x+y)-f(x) \right).
\end{align*}
\vbox{
Let us fix $C>0$ and choose $N\in \NN$ so that $N>4 C^2$ and let $n_k=n_k(N)$ be a strictly increasing sequence with $m_{n_k}=N$ ($k=1,2,\dots$). Now with $x_k=2^{n_k}$ and $y_k=\frac{1}{N} 2^{n_k}=|I_{n_k}|$, we have $f(x_k+y_k)-f(x_k) = -1/\sqrt{N}$, thus
$$
F(x_k+y_k) - F(x_k) +C y_k \leq - \frac{x_k+y_k}{\sqrt{N}} + C \frac{x_k}{N} \leq - \frac{x_k}{\sqrt{N}} + \frac{\sqrt{N}}{2} \frac{x_k}{N} = - \frac{x_k}{2\sqrt{N}},
$$
because $C<\sqrt{N}/2$. Therefore, choosing $c:=1/(2\sqrt{N})$ yields the assertion.}
\end{proof}

\begin{remark}\label{weakas}
With Corollary \ref{cor:dividebylog}, we proved that under the assumption $F(x)=O(x \log^m x)$, the assumption $F(x)/x$ slowly (respectively boundedly) decreasing implies that $F(x)/(x \log^{m-1}x)$ is slowly (respectively boundedly) decreasing.
\end{remark}

%%%%%%%%%%%%%%%%%%%%%%%%%%%%%%%%%%%%%%%%%%%%%%%%%%%%%%%%%%%%%%%%%%%%%%%%%%%%%%%%%%%%%%%
\section{Preliminary lemmas} \label{preliminary}

We recall here a few classical tools in analysis and formulate some technical lemmas.

We define, for $T>0$ and $u\in\reel$, the usual Fej\'er kernels as
\begin{equation}\label{eq:chiFejer}
\chi(u):=\frac{1}{2\pi}\left(\frac{\sin({u}/{2})}{{u}/{2}}\right)^2\quad\mbox{ and } \quad \chi_T(u):=T \chi(Tu).
\end{equation}
These kernels satisfy the following properties:
\begin{equation}\label{intfejer}
\int_\RR\chi_T(u)du=1,
\end{equation}
\begin{equation}\label{fourierfejer}
\widehat{\chi_T}(\tau)=\widehat{\chi}(\tau/T)=(1-|\tau|/T)_{+},
\end{equation}
and for $q>0$,
\begin{equation}\label{eq:Iesti}
I_q:=\int_{|u|>q/T} \chi_T(u) du=\int_{|u|>q}\chi(u) du \leq \frac{4}{\pi q}.
\end{equation}

We now state a lemma which will come in handy in Sections \ref{sec:boundedness} and \ref{sec:multiplepoles}.
\begin{lemme}\label{lemf}
Let $T>0$ be a positive constant and $f:\RR\rightarrow \RR$ be a bounded function of locally bounded variation.
Let $\lambda_1>0$, $\lambda_2\ge 0$ and $\lambda_3\ge 0$ be some %%%%positive
constants and assume that $f$ satisfies
\begin{equation}\label{condlemf}
f(u+v)\geq \lambda_1 f(u) \mbox{ as soon as }f(u)>\lambda_2 \mbox{ and } v\in[0,10/T].
\end{equation}
Assume furthermore that
\begin{itemize}
\item either $f(u)\geq -\lambda_3$ for any $u\in\RR$ (Case 1),
\item or $\lambda_1>\dfrac{4}{5\pi-4}$ and the function $g$ defined on $\RR$ by $g(u)=-f(-u)$ satisfies \eqref{condlemf} (Case 2).
\end{itemize}
Then we have
$$\|f\|_\infty\leq k_1\int_{-T}^{T} \left\vert \widehat{f}(\tau) \right\vert d\tau+k_2$$
with
$$ 
k_1:=\frac{5}{2(5\pi-4)\lambda_1}, \qquad \textrm{and} \qquad k_2:=\max\left(\lambda_2,\frac{4 \lambda_3}{(5\pi-4)\lambda_1},\lambda_3\right).
$$
in the first case and
$$
k_1:=\frac{5}{2(5\pi-4)\lambda_1-8}, \qquad \textrm{and} \qquad k_2:=\lambda_2.
$$
in the second one.
\end{lemme}

\begin{proof}
In both cases, we may assume that $\|f\|_\infty=\sup_uf(u)$. Otherwise the lower bound  $f(u)\geq -\lambda_3$ would give the result in the first case (as $\lambda_3\leq k_2$) and in the second case, we could consider $g$ defined by $g(u)=-f(-u)$ instead of $f$. Furthermore, if $\|f\|_\infty \leq \lambda_2$, we are done as $\lambda_2\leq k_2$ in both cases.

So for the rest of the proof we assume, as we may, that $\|f\|_\infty=\sup_uf(u)>\lambda_2$. For $u\in\RR$, we estimate the integral
$$J(u):=\int_\RR \chi_{T}(v) f(u+ 5/T+v) dv,$$
where $T>0$ and $\chi_T$ is defined in \eqref{eq:chiFejer}.
Plancherel's formula and \eqref{fourierfejer} lead to the upper bound
\begin{align}\label{majJ}
|J(u)|&=\left|\frac{1}{2\pi}\int_{-\infty}^{\infty} \widehat{\chi_{T}}(\tau) \widehat{f}(\tau) e^{i( 5/T+u)\tau} d\tau\right|\leq \frac{1}{2\pi}\int_{-T}^{T} \left\vert \widehat{f}(\tau) \right\vert d\tau.
\end{align}

The integral $J(u)$ is real, and cutting the integral to parts over $[-5/T,5/T]$ and outside that, we find the lower estimation
\begin{equation}\label{minJ_gen}
J(u)\geq (1-I_5) \min_{v\in [0,10/T]} f(u+v) + I _5 \min_{t\not\in[u,u+10/T]} f(t),
\end{equation}
where $I_5$ is defined in \eqref{eq:Iesti}.

Now for any $0<\varepsilon<1-\la_2/\|f\|_\infty$ we take some $u\in \RR$ with $(\la_2<)$ $(1-\ve) \|f\|_\infty < f(u) \leq \|f\|_\infty$. For such a $u$, we have $f(u+v) \geq \lambda_1 f(u)$ for all $v\in[0,10/T]$ by assumption, thus \eqref{minJ_gen} gives
\begin{align}\label{minJ2_1}
J(u)\geq (1-I_5) \lambda_1f(u) + I _5 \min_{\RR} f
\end{align}
In Case 1, we use $\min_{\RR} f \geq -\lambda_3$ whereas in Case 2 we use $\min_{\RR} f \geq -\|f\|_\infty>-f(u)/(1-\ve)$.
Combining this with \eqref{minJ2_1} and \eqref{majJ}, we are led to
$$
f(u)\leq  \lambda_4 \int_{-T}^{T} \left\vert \widehat{f}(\tau) \right\vert d\tau+\lambda_5
$$
with
$$\left\{\begin{array}{ll}
\lambda_4&=\frac{1}{2\pi} ((1-I_5)\lambda_1)^{-1}\leq 5\left(2(5\pi-4)\lambda_1\right)^{-1}, \\ \lambda_5&=\lambda_3I_5 ((1-I_5)\lambda_1)^{-1}\leq 4(5\pi-4)^{-1} (\lambda_3/\lambda_1)
\end{array}\right.$$
in Case 1 and in Case 2
$$\left\{\begin{array}{ll}
\lambda_4&=\frac{1}{2\pi} ((1-I_5)\lambda_1-I_5/(1-\varepsilon))^{-1}\leq 5\left(2(5\pi-4)\lambda_1-8/(1-\varepsilon)\right)^{-1},\\ \lambda_5&=0
\end{array}\right.$$
where $\varepsilon$ is taken small enough so that $(5\pi-4)\lambda_1(1-\varepsilon)>4$. (Recall that in Case 2, $\lambda_1>4/(5\pi-4)$ by condition).

Taking into account also the already settled case when $\|f\|_\infty<\lambda_2$, and letting $\ve$ tend to $0$, we therefore get the announced inequalities.
\end{proof}
Now we introduce some integrals which we shall need in Sections \ref{sec:boundedness} and \ref{sec:multiplepoles}. Let $m, T$ be positive real numbers and $\si$ be a nonnegative real number. We consider the following quantities.
\begin{equation}\label{defwallis}
W_m(\si,T):=\int_{0}^T \frac{dt}{|\si+it|^{m}},\quad Z_m(\si,T):=\int_{0}^T \frac{dt}{|\si+it|^{m}|1+\si+it|}.
\end{equation}
The integrals $W_m(\si,T)$ satisfy the following properties:
\begin{equation}\label{Wsi}
W_m(\si,T)=\si^{1-m}W_m(1,T/\si);\qquad
\end{equation}
\begin{equation}\label{estWallis}
W_m(\si,T)\ll_m \begin{cases} \si^{1-m}& \mbox{ if }m>1\\ \log\left(1+ \frac{T}{\si}\right)& \mbox{ if }m=1\\ T\{\max(T,\si)\}^{-m}& \mbox{ if }m<1\end{cases}.
\end{equation}
When $m>1$, we also have
\begin{equation}\label{estwallis2}
W_m(1,T) < W_m(1,\infty) < \frac{m}{m-1}.
\end{equation}
Analogously, a standard calculation furnishes for $0<\si<1$ (with $\asymp$ standing for $\ll$ and $\gg$ together)
\begin{equation}\label{estZm}
Z_m(\si,1)   \asymp_m Z_m(\si,\infty)   \asymp_m \begin{cases} \si^{1-m}& \mbox{ if } ~ ~ m>1 \\ \log(1+1/\si)& \mbox{ if}~~ m=1\\ 1 & \mbox{ if} ~~ 0<m<1\end{cases}
\end{equation}
and
\begin{equation}\label{estZ0}
Z_0(\si,T)   \asymp\log(1+T).
\end{equation}

\bigskip

We finally state a lemma providing some upper bounds for a certain function $\beta$ and its differences.
\begin{lemme}\label{l:beta} For any $\omega>-1$ and $t\in\RR$, let $\beta(\omega,t)$ be defined by
\begin{equation}\label{eq:betadef}
\beta(\omega,t):=\left\{\begin{array}{ll} \frac{t^{\omega}}{\Gamma(\omega+1)}e^{-t}(1-e^{-t})&(t>0)\\0&(t\leq0).\end{array}\right.
\end{equation}
Then for $\omega>-1$, $x\in\RR$ and $y>0$, we have
\begin{equation}\label{eq:betaest}
\left\{\begin{array}{l} |\beta(\omega,x)| \leq 1/\sqrt{\pi}, \\|\beta(\omega,x+y)-\beta(\omega,x)|\leq y^{\om+1}+2y /\sqrt{\pi}. \end{array}\right.
\end{equation}
\end{lemme}

\begin{proof}
In all of this argument we will repeatedly use that for any $\xi>0$
\begin{equation}\label{max1}
\max_{t>0} t^\xi e^{-t} = \xi^\xi e^{-\xi},
\end{equation}
 and (see \cite[\S 12.33]{WW})
 \begin{equation}\label{mingamma}
 \Gamma(\xi) > \xi^{\xi-1/2}{e}^{-\xi} \sqrt{2\pi}.
 \end{equation}
 We also recall the functional equation of $\Gamma$: for any $\xi\not\in\ZZ^{-}$, $\Gamma(\xi+1)=\xi\Gamma(\xi)$.
 Let $\omega>-1$ and $t$ be real numbers. First we give an upper bound for $|\beta(\omega,t)|$.

 If $t\leq 0$, then $\beta(\omega,t)=0$ so we will restrict to $t>0$, where $\beta(\om,t)>0$ and it suffices to estimate $\beta(\omega,t)$ from above.
\begin{itemize}
\item If $-1<\omega\leq 1$, we use $1-e^{-t}<t$ and \eqref{max1} and \eqref{mingamma} with $\xi=\omega+1$. We obtain
$$\beta(\om,t) <\frac{ t^{\om+1} e^{-t}}{\Gamma(\om+1)}\leq \frac{ (\om+1)^{\om+1} e^{-\om-1} }{ (\om+1)^{\om+1/2} e^{-\om-1} \sqrt{2\pi}} =\frac{\sqrt{\om+1}}{\sqrt{2\pi}} \leq \frac{ 1}{ \sqrt{\pi}}.$$
\item If $\omega>1$, we use $1-e^{-t}<1$ and \eqref{max1} and \eqref{mingamma} with $\xi=\omega$. We obtain
$$\beta(\om,t)<\frac{t^{\om}e^{-t}}{\Gamma(\om+1)}\leq \frac{ \om^\om e^{-\om} }{ \om \Gamma(\om)} \leq \frac{\om^\om e^{-\om} }{
\om^{\om+1/2} e^{-\om} \sqrt{2\pi}} = \frac{1}{\sqrt{2\pi\om}} < \frac{1}{\sqrt{2\pi}}.$$
\end{itemize}

Now consider the change of the function, hence the derivative of $\beta$ for $t>0$:
$$
\beta'(\om,t)=\frac{[2te^{-t}-t+\om(1-e^{-t})] t^{\om-1} e^{-t}}{\Gamma(\omega+1)}.
$$
\begin{itemize}
\item
For $-1<\om\leq 0$, we use $t(e^{-t}-1)\leq 0$ first and then $1-e^{-t}>te^{-t}$ and $e^{-t}<1$ to infer
$$
\beta'(\om,t)<\frac{[te^{-t}+\om(1-e^{-t})] t^{\om-1} e^{-t}}{\Gamma(\omega+1)}<\frac{(1+\om)t^{\om}}{\Gamma(\om+1)} \leq (1+\om) t^\om,
$$
as $\Gamma(\xi)\geq1$ when $0<\xi<1$. For $x\in\RR$ and $y>0$ this yields
$$
\beta(\omega,x+y)-\beta(\omega,x) \leq {1+\om}\int_{\max(0,x)}^{\max(0,{x+y})} t^{\om} dt < {1+\om} \int_0^{y}  t^{\om} dt \leq y^{\om+1}.
$$
Next, for a lower estimation using $te^{-t}>0$ and $1-e^{-t}<t$ we obtain
\begin{align*}
\beta'(\om,t)
&> \frac{-t(1-e^{-t})+\om t}{\Gamma(\omega+1)}t^{\om-1} e^{-t}>  \frac{(-t+\om )t^{\om} e^{-t}}{\Gamma(\om+1)}.
\end{align*}
As before, integration yields  for $x\in \RR$ and $y>0$
\begin{align*}
\beta(\omega,x+y)-\beta(\omega,x)
&\geq\frac{1}{\Gamma(\om+1)}\left(\frac{\om}{\om+1} y^{\om+1} - \left(\frac{\om+1}{e}\right)^{\om+1} y\right)
\end{align*}
where we used \eqref{max1} to estimate the last integral.
For the first coefficient, using that $\Gamma(\xi)$ is already increasing for $\xi\geq 2$, we get
$$\frac{\om}{(\om+1)\Gamma(\om+1)}=\frac{\om}{\Gamma(\om+2)}=\frac{\om(\om+2)}{\Gamma(\om+3)} > \frac{-1}{\Gamma(\om+3)}> \frac{-1}{\Gamma(2)}=-1.$$
%%%%%% $1/\min{1<\xi<2} \Gamma(\xi) = ???$.
For the second coefficient, we apply \eqref{mingamma} with $\xi=\om+1$ and get
$$\frac{1}{\Gamma(\om+1)} \left(\frac{\om+1}{e}\right)^{\om+1}\leq \sqrt{\frac{1+\om}{2\pi}}\leq\frac{1}{\sqrt{2\pi}}.$$
Whence for $-1<\om\leq 0$ we get $|\beta(\omega,x+y)-\beta(\omega,x)| \leq y^{\om+1} + y/\sqrt{2\pi}$.

\item Consider now the case $\om\geq 2$.
We use $(e^{-t}-1)<0$ and $1-e^{-t}<1$ and then we use \eqref{max1} and \eqref{mingamma}. We get
\begin{align*}
\beta'(\om,t)&<\frac{[te^{-t}+\om ] t^{\om-1} e^{-t}}{\Gamma(\omega+1)}<\frac{ t^{\om} e^{-2t}}{\om\Gamma(\omega)}+\frac{ t^{\om-1} e^{-t}}{(\om-1)\Gamma(\omega-1)}\\
&\leq 2^{-\om}\frac{1}{\sqrt{2\pi\om}}+\frac{1}{\sqrt{2\pi(\om-1)}}\leq\left(\frac{2^{-2}}{2}+\frac{1}{\sqrt{2}}\right)\frac{1}{\sqrt{\pi}}\leq\frac{1}{\sqrt{\pi}}.
\end{align*}
 This entails
$\beta(\omega,x+y)-\beta(\omega,x) < y/\sqrt{\pi}$.

From the other direction, we have
\begin{align*}
\beta'(\om,t)&\geq-\frac{e^{-t} t^{\om}}{\om\Gamma(\omega)}\geq -\frac{1}{\sqrt{2\pi\om}}\geq -\frac{1}{2\sqrt{\pi}}\end{align*}
thus
$\beta(\omega,x+y)-\beta(\omega,x) \geq -y/(2\sqrt{\pi})$.\\
Whence for $\om\geq 2$ we get $|\beta(\omega,x+y)-\beta(\omega,x)| \leq  y/\sqrt{\pi}$.

\item Finally, let us assume that $0<\om<2$.
Similarly as before, we use \eqref{max1} and \eqref{mingamma} to get for $t>0$
\begin{align*}
\beta'(\om,t)&\leq\frac{[te^{-t}+\om t] t^{\om-1} e^{-t}}{\Gamma(\omega+1)}=\frac{ t^{\om} e^{-2t}}{\Gamma(\omega+1)}+\frac{ t^{\om} e^{-t}}{\Gamma(\omega)}\\
&\leq\frac{ t^{\om} }{\Gamma(\omega+1)}+\sqrt{\frac{ \om}{2\pi}}\leq \frac{ (\om+1)t^{\om} }{\Gamma(\omega+2)}+\frac{1}{\sqrt{\pi}} \leq (\om+1)t^{\om}+\frac{1}{\sqrt{\pi}}.
\end{align*}
This entails
$\beta(\omega,x+y)-\beta(\omega,x) \leq y^{\omega+1} + y/\sqrt{\pi}$.
We also have
\begin{align*}
\beta'(\om,t)&\geq -\frac{ t^{\om+1} e^{-t}}{\Gamma(\omega+1)}\geq -\sqrt{\frac{\omega+1}{2\pi}}> -\frac{2}{\sqrt{\pi}}
\end{align*}
yielding
$\beta(\omega,x+y)-\beta(\omega,x) \geq -2y/\sqrt{\pi}$.
\end{itemize}
\end{proof}

%%%%%%%%%%%%%%%%%%%%%%%%%%%%%%%%%%%%%%%%%%%%%%%%%%%%%%%%%%%%%%%%%%%%%%%%%%%%%%%%%%%%%%%%%%%%%%%%%%%%%%%%%%%%%%%%%%%%%%%%%%%%%

\section{Boundedness derived from local assumptions on the Mellin-Stieltjes transform}\label{sec:boundedness}

We can find in the book of Korevaar, see \cite[Theorem III.10.1, p. 142 and Proposition III.10.2, p. 143]{Ko} some convergence and boundedness theorems involving general Tauberian conditions.
These results are of two kinds. If the Mellin-Stieltjes transform of $A$ is regular enough throughout the boundary, we get a convergence theorem, whereas if we have local regularity of the Mellin-Stieltjes transform near $s=1$, we only get a boundedness theorem. In this section we work out a more general "local theorem" allowing the Mellin-Stieltjes transform to have some singularities at $s=1$.

%%%\begin{theoreme}\label{p:Korevaar}
%%%Let $A:\RR\to \RR$ be a function of locally bounded variation vanishing on $(-\infty,1)$ such that
%%%\begin{enumerate}
%%%\item either $A(x)/x$ satisfies the bounded decrease condition of Definition \ref{def:Korevaar},
%%%\item or $m>2$ and denoting $\mu=(m-1)_+$ we have that $A(x)/(x\log^{\mu-1}{x})$ is of bounded decrease according to Definition \ref{def:Korevaar},
%%%\item or $A$ satisfies the moderate decrease condition of Definition \ref{def:moderatedec}.
%%%\end{enumerate}
%%%
%%%Assume further that the Mellin-Stieltjes transform $\AAA(s):=\int_1^\infty x^{-s}dA(x)$ is convergent in the half plane $\Re{s}>1$ and satisfy for some $T>0$ and some real $m$,
%%%$$|\AAA(s)|\ll \frac{1}{(s-1)^{m}}\quad (1<\Re(s)<2,\quad |\Im(s)|\leq T).$$
%%%
%%%Then for $x\geq e$, we have $A(x) = O( x\log^{\mu}x) $ if $m\not=1$ and $A(x)=O(x\log\log{x})$ if $m=1$.
%%%\end{theoreme}

\begin{theoreme}\label{th:kleTfixe}
Let $m \geq 0$ and $T>0$ be fixed parameters and let $A:\RR\to \RR$ be a function of locally bounded variation vanishing on $(-\infty,1)$ such that
\begin{enumerate}
\item either we have that $A(x)/x$ is of bounded decrease on $[1,\infty)$ according to Definition \ref{def:Korevaar},
\item or we have that $A(x)/(x\log^{(m-2)_{+}}{x})$ is of bounded decrease on $[e,\infty)$ according to Definition \ref{def:Korevaar},
\item or $A$ satisfies the moderate decrease condition of Definition \ref{def:moderatedec} on $[1,+\infty)$.
\end{enumerate}

Assume further that the Mellin-Stieltjes transform $\AAA(s):=\int_1^\infty x^{-s}dA(x)$ converges in the half plane $\Re{s}>1$ and satisfy for all $0<\si<1$
\begin{equation}\label{eq:integralcondionA}
\int_0^{T} \frac{|\AAA(1+\si+it)|}{|1+\si+it|}dt \leq K_1 \left\{\begin{array}{ll}Z_m(\si,\infty) & \mbox{ if }m>0,\\Z_0(\si,T) & \mbox{ if }m=0\end{array}\right.
\end{equation}
where $Z_m$ is defined in \eqref{defwallis} and $K_1$ is some positive constant. %%% depending on $m$.

Then we have $A(x) = O( x\log^{(m-1)_{+}}x)$ for all $x\geq e$ if $m\not=1$ and $A(x)=O(x\log\log{x})$ for all $x\geq e^e$ if $m=1$, where the implicit constant of the $O$ symbol depends explicitly on $m$, on $K_1$ and
\begin{itemize}
\item on $\left\|A(x)/x\right\|_{1\leq x \leq e^{1+50/T}}$ and on $\overline{\nu}(A(x)/x;e^{10/T})$ under Condition {\it 1},
\item on $\left\|A(x)/x\right\|_{1\leq x \leq e^{1+c/T}} $ with $c=10\max(5,1+(m-2)_+/\log{2})$ and on $\overline{\nu}(A(x)/(x\log^{(m-2)_{+}}x);e^{10/T})$ under Condition {\it 2},
\item on $\left\|A(x)/x\right\|_{1\leq x \leq e}$, on $T$ (more precisely, on $e^{10/T}$) and on the constants $B_1$, $B_2$ under Condition {\it 3}.
\end{itemize}
In case $m=0$, we also have a dependence on $T$ with both the boundedly and the moderately decreasing conditions.
\end{theoreme}

\begin{remark}\label{rem:retro}
In retrospect we can observe that \emph{under the condition of \eqref{eq:integralcondionA}} the first assumption, i.e. that  "$A(x)/x$ is boundedly decreasing", is stronger (more restrictive) than the second one stating "$A(x)/(x\log^{{(m-2)}_{+}}{x})$ is boundedly decreasing". Indeed, if the first assumption is satisfied, then the above theorem furnishes the $O$-result which in turn entails this second assumption as noticed in Remark \ref{weakas}.
\end{remark}

\begin{proof}
Concerning the bounded decrease condition, the case $m=0$ is given in the book of Korevaar, see \cite[Proposition III.10.2, (i) p. 143]{Ko}. We will generalize his arguments in this proof.\\
For $\sigma>0$, we define the function $h_\sigma$ on $\RR$ by $h_\sigma(t):= e^{-(1+\sigma) t}A(e^t)$. According to Remark \ref{l:convergenceetordre}, the function $h_{\sigma}$ is bounded.
To prove the theorem, it is enough to prove with $\mu:=(m-1)_{+}$ and some $\si_0\in(0,1)$ that it holds
\begin{equation}\label{boundh}
\|h_{\si}\|_{\infty}\ll \left\{\begin{array}{cc} \sigma^{-\mu} &\mbox{ if }m\not=1\\ \log(1+1/\sigma) & \mbox{ if }m=1 \end{array}\right. \quad (0<\sigma<\sigma_0 ).
\end{equation}
Indeed, with $t:=\log x$, we have $A(x)=h_{1/t}(t)e^{t+1}$, thus  \eqref{boundh} implies for $x>e^{1/\si_0}$,
$$A(x)\ll \begin{cases}  x\log^{\mu}x &\mbox{ if }m\not=1\\ x\log\log{x}& \mbox{ in case }m=1\end{cases},$$
while for $x\leq e^{1/\si_0}$, local boundedness of $A$ ensures the statement provided that we allow a dependence of the implied constant on $\|A(x)/x\|_{[1,e^{1/\si_0}]}$.

For a parameter $\lambda\geq0$, to be fixed later, we define the sets $S_\mu^{\pm}(\lambda)$ as
$$
S_\mu^{+}(\lambda):=\{ t\geq 0 ~:~ A(e^t) >\lambda e^t t^{\mu}\}, \quad S_\mu^{-}(\lambda):=\{ t\geq 0 ~:~ A(e^t) <-\lambda e^t t^{\mu}\}.
$$
If $t\geq 0$ and $t\not\in S_\mu(\lambda):=S_\mu^+(\lambda)\cup S_\mu^-(\lambda)$, then with the convention $0^0=1$
$$
| h_\si(t)|\leq  \lambda t^{\mu} e^{-\si t}\leq \frac{\lambda}{\si^{\mu}}\|e^{-v}v^{\mu}\|_{\RR^+}=\frac{\lambda}{\si^{\mu}} \left(\frac{\mu}{e}\right)^{\mu} \left( \ll_{\mu,\lambda} \si^{-\mu} \right).
$$
Now we put $\mu_0=(\mu-1)_{+}=(m-2)_{+}$ under Condition {\it 2} and $\mu_0=0$ under Conditions {\it 1} and {\it 3}.

\begin{itemize}
\item Under Condition {\it 1} and {\it 2}, we put
\begin{equation}\label{deflambda2}
\lambda_2:=\max\left\{ \left\|\frac{A(x)}{x}\right\|_{\left[1,\exp(1+\frac{10}{T} \max\{5,\frac{\mu_0}{\log 2}+1\})\right]} , ~ \lambda \left(\frac{\mu}{e}\right)^{\mu} \si^{-\mu} \right\}.
\end{equation}
Let $u$ be a real number such that $|h_\sigma(u)|>\lambda_2$.\\
In view of $|h_\si(u)|<|A(e^u)|/e^u$, we necessarily have $u\geq 1+50/T$. Moreover, if $\theta$ is the sign of $h_\sigma(u)$, then $u\in S_\mu^\theta(\lambda)$ and, as obviously $S_{\mu}^\theta(\lambda)\cap[1,+\infty)\subset S_{\mu'}^\theta(\lambda)$  for any $\mu'\leq \mu$, we have as well $u\in S_{\mu_0}^\theta(\lambda)$.

%%%% We shall apply Lemma \ref{lemf} to $h_\si$.

By condition of bounded decrease we have for $e \leq x\leq y\leq e^{10/T} x$
\begin{equation}\label{xy}
\frac{A(y)}{y(\log y)^{\mu_0}} - \frac{A(x)}{x(\log x)^{\mu_0}} \geq  -\overline{\nu}(e^{10/T}),
\end{equation}
where $\overline{\nu}(e^{10/T}):= \overline{\nu}(A(x)/(x\log^{\mu_0}x);e^{10/T})$.\\
Let $v$ be a real number satisfying $0\leq v\leq 10/T$.
We apply \eqref{xy} with $(x,y)=(e^u,e^{u+v})$ if $h_\si(u)>0$ or $(x,y)=(e^{u-v}, e^u)$ if $h_\si(u)<0$. Note that $u\geq 1+50/T>1$ and $u-v \geq 1+40/T>1$, so in case $\lambda>0$, we infer for any $u$ and $v$ as above
\begin{align*}
\theta h_\si(u+\theta v)&=\theta \frac{(u+\theta v)^{\mu_0}}{e^{\si(u+\theta v)}}\left(\frac{A(e^{u+\theta v})}{e^{u+\theta v}(u+\theta v)^{\mu_0}}-\frac{A(e^{u})}{e^{u}u^{\mu_0}}+\frac{A(e^{u})}{e^{u}u^{\mu_0}} \right)\\
& \geq \theta e^{-\si(u+\theta v)}\left( -\theta \overline{\nu}(e^{10/T}) +\frac{A(e^{u})}{e^{u}u^{\mu_0}} \right){(u+\theta v)^{\mu_0}}\\
& \geq \theta e^{-\si(u+\theta v)}\left(1-\frac{\overline{\nu}(e^{10/T})}{\lambda}\right) \frac{A(e^u){(u+\theta v)^{\mu_0}}}{e^{u}u^{\mu_0}}\\
& \geq e^{-\theta \si v}\left(1-\frac{\overline{\nu}(e^{10/T})}{\lambda}\right)|h_\sigma(u)|\left(\frac{u+\theta v}{u}\right)^{\mu_0}.
\end{align*}
Now we chose the parameters $\lambda$ and $\si_0$ as $\lambda:=5 \overline{\nu}(e^{10/T})$ and $\si_0:=\min(1,\frac{\log 2}{10} T)$. \\
When $\theta=+$, we obtain from the above for any $0<\si <\si_0$ the estimate
\begin{equation}\label{minh+}
h_\si(u+v) \geq \frac45 e^{-\si 10/T} h_\sigma(u)\geq \frac25 h_\sigma(u).
\end{equation}
Similarly, for $\theta=-$, using $u > (1+\mu_0/\log 2)10/T$ we obtain for $\si>0$
\begin{equation}\label{minh-}
- h_\si(u-v) \geq e^{\si v} \frac45 |h_\sigma(u)| \left(\frac{u-10/T}{u} \right)^{\mu_0}  \geq
%%%\frac25 |h_\sigma(u)|= 
\frac25 \left(-h_\sigma(u)\right),
\end{equation}
because $1+w\leq e^w$ implies $\frac{q+1}{q}\leq \exp(\frac{1}{q})$ and $\frac{q}{q+1}\geq \exp(\frac{-1}{q})$ which  yields
$$
\left(\frac{u-10/T}{u} \right)^{\mu_0}  \geq \left(\frac{\mu_0/\log 2}{\mu_0/\log 2 +1} \right)^{\mu_0} \geq \exp\left(\frac{-\mu_0}{\mu_0/\log 2} \right)= \frac12.
$$
Inequalities \eqref{minh+} and \eqref{minh-} remain true when $\overline{\nu}(e^{10/T})=0$ and hence $\lambda=0$.\\
We have just proved that if $0<\si<\si_0=\min(1,\frac{\log 2}{10} T)$, then for any $u$ such that $\theta h_\si(u)>\lambda_2$ and any $v\in[0,10/T]$, we have $\theta h_\si(u+\theta v)\geq \frac25 \theta h_\si(u)$.
So choosing $\lambda_1:=2/5$ and noting that $2/5> 4/(5\pi-4)$, an application of the second part of Lemma \ref{lemf} to $h_\si$ leads to
\begin{equation}\label{eq:hsigmatransf}
\|h_\si\|_\infty\ll \int_{-T}^{T} \left\vert \widehat{h_\si}(\tau) \right\vert d\tau+\lambda_2 \qquad (0<\si<\si_0),
\end{equation}
with an absolute implicit constant.

\item Assume now that Condition {\it 3} holds. We take  $\lambda>0$.%%%%(\mu e^{-1})^{\mu}\si^{-\mu}$.
\\
Let $u$ be a real number such that $h_\sigma(u)>\lambda$, then $u\in S_0^+(\lambda)$.
By condition of moderate decrease, we have for any $v\geq 0$
$$
A(e^{u+v})-A(e^u) \geq-B_1(e^v-1)e^u-B_2\varphi(e^u) e^u.
$$
If $v\in[0,10/T]$ and $\lambda>0$, we can infer
\begin{align*}
h_\si(u+v)&=e^{-\si(u+v)}\left( \frac{A(e^{u+v})-A(e^u)}{e^{u+v}}+e^{-v}\frac{A(e^u)}{e^{u}}\right)
 \\& \geq e^{-\si(u+v)}\left( -B_1(1-e^{-v})-B_2e^{-v}\varphi(e^u)  +e^{-v}\frac{A(e^u)}{e^{u}}\right)\\
 &\geq e^{-\si(u+v)-v}\left( -\max(B_1,B_2)e^{10/T} + \frac{A(e^u)}{e^{u}}\right)\\
 & \geq e^{-(\si+1) v}\left(1-\frac{\max(B_1,B_2)e^{10/T}}{\lambda}\right) h_\sigma(u).
\end{align*}
Choosing $\lambda:=5\max(B_1,B_2)e^{10/T}$, we get for $0<\si<1$ and $u$ and $v$ as above
\begin{equation}\label{minhplus}
h_\si(u+v)\geq \frac45e^{-(\si+1) 10/T} h_\sigma(u)
\geq \frac45e^{-20/T} h_\sigma(u).
\end{equation}
This last inequality remains valid when $\lambda=0$, too.

Now we use \eqref{minF} to see that $h_{\si} (u)>-B$ with $B:=B(A,1)$. Finally, we apply the first part of Lemma \ref{lemf} with $\lambda_2=\lambda$,%%%(\mu e^{-1})^{\mu}\si^{-\mu}$, 
$\lambda_1=\frac45e^{-20/T} $ and $\lambda_3=B$. Thus we get for $0<\si<1$,
$$
\|h_\si\|_\infty\ll \int_{-T}^{T} \left\vert \widehat{h_\si}(\tau) \right\vert d\tau+\lambda_2
$$
where the implicit constant depends on $e^{20/T}$ and $B$.
\end{itemize}

Since in all cases $\la_2\ll \si^{-\mu}$, it remains to estimate the integral $\int_{-T}^{T} \left\vert \widehat{h_\sigma}(\tau) \right\vert d\tau$. As
\begin{equation}\label{eq:handA}
\widehat{h_{\sigma}}(\tau)=\frac{\AAA(1+\si+i\tau)}{(1+\si+i\tau)},
\end{equation}
using the hypothesis \eqref{eq:integralcondionA} on $\AAA$, we get
$$
\int_{-T}^{T} \left\vert \widehat{h_\sigma}(\tau) \right\vert d\tau =\int_{-T}^{T}  \left|\frac{\AAA(1+\si+i\tau) }{1+\si+i\tau}\right|d\tau \ll \left\{\begin{array}{ll}Z_m(\si,\infty) & \mbox{ if }m>0,\\Z_0(\si,T) & \mbox{ if }m=0.\end{array}\right.
$$
Applying \eqref{estZm} and \eqref{estZ0} gives \eqref{boundh}, whence the result.
\end{proof}

We shall now discuss the condition \eqref{eq:integralcondionA} on $\AAA$. First, we give some more natural but stronger  condition under which the conclusion holds.

\begin{corollary}\label{c:meromorphiclocal}
Under the same conditions than in Theorem \ref{th:kleTfixe} but replacing \eqref{eq:integralcondionA} by the fact that the Mellin-Stieltjes transform $\AAA$ has a meromorphic continuation in a neighborhood of $s=1$, with a pole of order $m \in \NN$ (i.e., a regular point if $m=0$) at $s=1$, we have $A(x) = O( x\log^{(m-1)_{+}}x) $ for all $x\geq e$.
\end{corollary}
\begin{proof}
Developing $\AAA$ in a neighborhood of $s=1$ gives for some $T>0$,
\begin{equation}\label{eq:ptwisecondionA}
|\AAA(s)|\ll \frac{1}{|s-1|^{m}}\quad (1<\Re(s)<2,\quad |\Im(s)|\leq T).
\end{equation}
 Clearly, the condition \eqref{eq:ptwisecondionA} entails \eqref{eq:integralcondionA}, hence the result follows except for $m=1$.\\
In case of a simple pole with residue $r$, however, we can apply this result to $A(x)-rx$ with $m=0$. To this note that together with $A(x)/x$ also $(A(x)-rx)/x$ is boundedly decreasing, and with $A(x)$ also $A(x)-rx$ is moderately decreasing, while the Mellin-Stieltjes transform of $A(x)-rx$ is $\AAA(s)-r/s$, that is regular. So under either assumptions on the controlled decrease of $A$, we get $A(x)-rx=O(x)$, thus $A(x)=O(x)$.
\end{proof}

%%%%%%%%%%%%%%%%%%%%%%%%%%%%%%%%%%%%%%%%%%%%%%%%%%%%%%%%%%%%%%%%%%%%%%%%%%%%%%%%%%
%%%%%%%%%%\bigskip
%%%%%%%%%%\bigskip
%%%%%%%%%%\bigskip
%%%%%%%%%%\centerline{* * * * *}
%%%%%%%%%%\bigskip
%%%%%%%%%%\bigskip
%%%%%%%%%%\bigskip

In the effective Wiener-Ikehara Theorem (Theorem \ref{ikin}), the condition to get an estimate for $A(x)$ is
\begin{equation}\label{eq:firstetadef}
\eta(\si,T):=\int_0^T \left|\frac{\AAA(1+\si+it)}{1+\si+it} - \frac{\AAA(1+2\si+it)}{1+2\si+it}\right| dt\ll \si^{-(m-1)_+}
\end{equation}
for some $T>0$ and for any $0<\si<1$.
For concrete Laplace transforms the expression in Condition \eqref{eq:integralcondionA} is larger (hence the condition is stronger), than $\eta(\si,T)$
because we usually find considerable cancelation in the latter. By the triangle inequality, if \eqref{eq:integralcondionA} holds, then we also have
$$\eta(\si,T)\ll \begin{cases}Z_m(\si,\infty)+Z_m(2\si,\infty)\asymp Z_m(\si,\infty)&\mbox{ if }m>0\\Z_m(\si,T)+Z_m(2\si,T) \asymp Z_m(\si,T)&\mbox{ if }m=0\end{cases}.$$
The following lemma states that the converse is also true.

\begin{lemma}\label{l:Aandetacompare}
Let $A:\RR\to \RR$ be a function of locally bounded variation vanishing on $(-\infty,1)$ and assume that the Mellin-Stieltjes transform $\AAA(s):=\int_1^\infty x^{-s}dA(x)$ converges in the half plane $\Re{s}>1$.\\
Assume further that for some fixed nonnegative value of $T$, there exist some positive constants $\mu$ and $k$ such that $\eta(\si,T)\leq k \si^{-\mu}$ for $0<\si<1$. \\
Then there exists a positive constant $K'$ such that for all $\si\in(0, 1)$ we have the estimate
\begin{equation}\label{eq:intAcompareeta}
\int_0^T \left|\frac{\AAA(1+\si+it)}{1+\si+it}\right| dt \leq K' Z_{\mu+1}(\si,\infty).
\end{equation}
If $\mu>0$, we have $K'=kk'$ with $k'$ depending explicitly only on $\mu$. \\
In case $\mu=0$ we can make the constant $K'$ explicit if either Condition \it{1} or {\it 2} for some $m>0$ or Condition \it{3} of controlled decrease in Theorem \ref{th:kleTfixe} holds. In this case the constant $K'$ depends explicitly on the parameters $k$, $T$ and $\AAA(2)$  and on the constant $M(A,1)$ under Condition {\it 1}, or on $M(A/\log^{(m-2)_+},e)$, $m$ and $\|A\|_{[1,e]} $ under Condition {\it 2}, or on $B(A,1)$ under Condition {\it 3}.
\end{lemma}
\begin{proof} First let $\mu>0$. By a repeated application of the triangle inequality and using \eqref{eq:handA} we obtain for arbitrary $N\in \NN$
\begin{align*}
\int_0^T \left|\frac{\AAA(1+\si+it)}{1+\si+it}\right| dt & \leq \sum_{n=0}^{N-1} \eta(2^n\si,T) +  \int_0^T \left|\frac{\AAA(1+2^N\si+it)}{1+2^{N}\si+it}\right| dt \\&\leq \sum_{n=0}^{N-1} k (2^n\si)^{-\mu} + \int_0^T |\widehat{h_{2^N\si}}(t)|dt
\\& \leq \frac{k}{1-2^{-\mu}}~{\si^{-\mu}} +T \|\widehat{h_{2^N\si}}\|_\infty .
\end{align*}
But for any $\theta>0$ we have
$$
\|\widehat{h_{\theta}}\|_\infty\leq  \|h_{\theta}\|_{1} = \int_{0}^{\infty} |A(e^u)|e^{(-1-\theta) u} du \ll \int_{0}^{\infty} e^{(1-\theta) u} du
$$
using Remark \ref{l:convergenceetordre} with $\ve=1$, so according to the monotone convergence theorem (or the Lebesgue convergence theorem) $\lim_{\theta\to \infty} \|{h_{\theta}}\|_1=0$ and we are led to
$$
\int_0^T \left|\frac{\AAA(1+\si+it)}{1+\si+it}\right| dt \leq \frac{k}{1-2^{-\mu}}~{\si^{-\mu}},
$$
which gives the assertion for $\mu>0$ in view of \eqref{estZm}.

In case $\mu=0$ we start our calculation similarly, but fix the value of $N$ as $N:=\lceil\log_2(1/\si)\rceil$
%(and so in particular as $0$ for $\si>2$),
and thus have $ 2^N \si \geq 1$, always. The above calculation then leads to
\begin{align*}
\int_0^T \left|\frac{\AAA(1+\si+it)}{1+\si+it}\right| dt & \leq kN + T \int_{0}^{\infty} |A(e^u)|e^{(-1-2^N\si) u} du \\& \leq k \left\{1+ \log_2\left(\frac{1}{\si}\right)\right\} + T \int_{0}^{\infty} |A(e^u)|e^{-2u}du.
\end{align*}
Clearly, the last integral is bounded in virtue of Remark \ref{l:convergenceetordre}, so at least with some ineffective constant $K'$ the assertion in \eqref{eq:intAcompareeta} follows. However, to make the constant explicit, we need to invoke the above explicit calculations with functions of controlled decrease.

Let us denote $\mu_0=0$ under Conditions {\it 1} and {\it 3} and $\mu_0=(m-2)_+$ under Condition {\it 2} and
$$k^{*}:=\left\{\begin{array}{ll}
M(A,1) & \mbox{ under Condition {\it 1}}\\
M(A/\log^{(m-2)_+},e)+\|A\|_{[1,e]} & \mbox{ under Condition {\it 2}}\\
 B(A,1) &\mbox{  under Condition {\it 3}.}
 \end{array}\right.$$

  By \eqref{minslowdec} and \eqref{minF}, we always have $A(x)>-k^{*}x(1+\log x)^{1+\mu_0}$ for $x\geq 1$ and so in view of $|A(e^u)|=A(e^u)_{+}+A(e^u)_{-}=A(e^u)+2A(e^u)_{-}$ we are led to
\begin{align*}
\int_{0}^{\infty} |A(e^u)|e^{-2 u}du & \leq \int_{0}^{\infty} A(e^u) e^{-2u}du +2k^{*} \int_0^\infty (u+1)^{1+\mu_0}e^{-u}du \\& \leq \int_{0}^{\infty} A(e^u) e^{-2 u}du +2k^{*} e\Gamma(2+\mu_0)
\\&=\frac{\AAA(2)}{2}+2e\Gamma(2+\mu_0)k^{*}=:K^{*},
\end{align*}
just denoting this last constant as $K^{*}$. So now we have for $0<\si<1$
\begin{align*}
\int_0^T \left|\frac{\AAA(1+\si+it)}{1+\si+it}\right| dt & \leq k\left(1+\log_2\left(\frac{1}{\si}\right) \right) + K^{*}T \\
&\leq \frac{2}{\log{2}} (K^{*}T+k) \log\left(1+\frac{1}{\si}\right).
%%%%%\\& \ll (K^{*}T+k) Z_1(\si,1).
\end{align*}
%Moreover, for $\si>2$ we have chosen $N=0$, so the first part of our estimation vanishes, and drawing from the above calculation directly we are led to
%\begin{align*}
%\frac{1}{T}\int_0^T \left|\frac{\AAA(1+\si+it)}{1+\si+it}\right| dt & \leq \int_{0}^{\infty} |A(e^u)|e^{(-1-\si) u} du \\& \leq \max_{1\leq x\leq e} |A(x)| \int_{0}^1 e^{-(1+\si) u}du + \int_{1}^{\infty} |A(e^u)| e^{-(1+\si) u}du
%\\& < \frac{\max_{1\leq x\leq e} |A(x)| }{1+\si} + e^{1-\si} K^{*} <K" \frac{1}{\si} \\ &\leq \frac{1/2}{\log(3/2)} K" \log\left(1+\frac{1}{\si}\right),
%\end{align*}
%with $K"\ll K^{*}+ \max_{1\leq x\leq e} |A(x)|$ with an implied absolute constant.

Taking into account $Z_1(\si,\infty)\asymp \log(1+1/\si)$ from \eqref{estZm} we get
$$
\int_0^T \left|\frac{\AAA(1+\si+it)}{1+\si+it}\right| dt \leq K^{**} \log\left(1+\frac{1}{\si}\right) \leq K' Z_1(\si,\infty),
$$
with $K^{**}$ and hence $K'$ depending explicitly on the parameters given.
\end{proof}

\begin{corollary}\label{c:localbddnessbyeta} Assume the same hypothesis as in Theorem \ref{th:kleTfixe}, but replacing \eqref{eq:integralcondionA} by
\begin{equation}\label{eq:integralconditioneta}
\eta(\si,T)\leq K_2 \si^{-(m-1)_+} \quad (0<\si<1)
\end{equation}
for some positive constant $K_2$.
Then we have
\begin{equation}\label{eq:semistrong}
|A(x)| \leq \begin{cases} K_3 x\log^{m-1} x & \textrm{if} ~~~ m>1 \\ K_3x\log\log x  &{\textrm{if}}~~~ 0<m\leq 1  \end{cases} \quad \left(x\geq e^e \right).
\end{equation}
Here the constant $K_3$ is explicit and depends on $|\AAA(2)|$ (if $m\leq 1$) and on the same parameters as in Theorem \ref{th:kleTfixe} but now dependence on $K_1$ in \eqref{eq:integralcondionA} is replaced by dependence on $K_2$ in \eqref{eq:integralconditioneta}.
\end{corollary}
\begin{proof}
We apply Lemma \ref{l:Aandetacompare} with $\mu=(m-1)_{+}$,  and thus obtain
\begin{equation}\label{eq:conditiontoZfrometa}
\int_0^T \left|\frac{\AAA(1+\si+it)}{1+\si+it}\right| dt \ll Z_{\mu+1}(\si,\infty).
\end{equation}
An application of Theorem \ref{th:kleTfixe} then gives the result.

Note that in case $\mu=0$, Condition {\it 2} also reduces to Condition {\it 1} and dependence of the implied constants on $m$ becomes out of consideration.
\end{proof}

%%%%%%%%%%%%%%%%%%%%%%%%%%%%%%%%%%%%%%%%%%%%%%%%%%%%%%%%%%%%%%%%%%%%%%%%%%%%%%%%%%

\section{Local lemma of "Ganelius-Tenenbaum type"}\label{sec:Ganelius}

In this section, we prove a new variant of the Lemma of Ganelius, (which is itself a further effective version of an inequality of Bohr). Ganelius' Lemma was used in various versions by Tenenbaum to obtain effective estimations related to the Wiener-Ikehara theorem and other Tauberian theorems. However, as it is seen from the version we quote below, the estimates worked in a "uniform" way, that is, only globally valid estimates could be used in the final results. Our goal will be to relax on this, that is to exploit asymptotic decrease of the estimates in the conditions. Succeeding in that provides a key to our approach in finding effective error terms in the Wiener-Ikehara Theorem even if only asymptotic negligibility conditions are available on the functions to be analyzed. That relaxation opens up the possibility to consider even not necessarily increasing functions $A(x)$, but also those who have only some relaxed asymptotic conditioning on their monotonicity (as discussed in Section \ref{sec:slowdecrease}).

\begin{lemme}[Ganelius-Tenenbaum]\label{lemmegalten}
Let $g:\reel\rightarrow\reel$ be a bounded function of $L^{1}(\reel)$ satisfying for some positive $T$
$$\sup_{x\leq y\leq x+1/T}(g(y)-g(x)) \leq K<\infty.$$
Then we have
$$\|g\|_\infty\leq 16 K+6\int_{-T}^{T}|\widehat{g}(\tau)|d\tau.$$
\end{lemme}
For a proof, see e.g. \cite{T2}, p. 328.

\begin{lemme}[Ganelius type lemma]\label{l:G} Let $\Phi$ be a nonnegative even function on $\reel$, nonincreasing on the half-line $[0,\infty)$.
Let $a>0$ and $T > 0$ be real numbers and $g: \reel \to \reel$ be a bounded measurable function satisfying
\begin{equation}\label{eq:i}
\sup_{0<y<1/T} \left\{ g(x+y)-g(x) \right\} \leq a + \Phi(x) \quad \left( x\geq x_1:=\frac{8}{T}\sqrt{1+\frac{\Phi(0)}{a}}\right)
\end{equation}
and
\begin{equation}\label{eq:ii}
    \widehat{g}(\tau)=0 \qquad (  |\tau|<T).
\end{equation}
Then we have
\begin{equation}\label{eq:Gan}
|g(x)|\leq  8\sqrt{\frac{14}{3}}\sqrt{(a+\Phi(0))(a+\Phi(x/2))}
\end{equation}
for all $x>x_0(T,a):=(16/T)\sqrt{1+\Phi(0)/a}$.
\end{lemme}

\begin{proof}
Our proof is a variation of the argument on page 327 of \cite{T2}.\\
Fix a real number $x>0$ such that $g(x)\not=0$ and denote $\theta:=g(x)/|g(x)|$. We define $ \chi_T$ as in \eqref{eq:chiFejer}.
We have $\textrm{supp}\widehat{\chi_T}=[-T,T]$, thus $\widehat{g} \widehat{\chi_T} =0$ and $g\ast \chi_T=0$.
Let $q>0$ be fixed.  By the
application of\eqref{eq:i} $(\lfloor 2q\rfloor +1)$-times, we get
\begin{align*}
0 & = \theta \int_{-\infty}^{\infty} g(x-t-\theta q/T) \chi_T(t) dt = \theta\int_{|t|>q/T}g(x-t-\theta q/T)\chi_T(t) dt\\
&+\theta \int_{-q/T}^{q/T} g(x) \chi_T(t)  dt   - \theta\int_{-q/T}^{q/T} (g(x)-g(x-t-\theta q/T)) \chi_T(t) dt  \notag\\
& \geq -\|g\|_{\infty} I_q+(1-I_q) |g(x)| - \left( \lfloor 2q \rfloor +1 \right) (1-I_q)\left(a +\max\limits_{[x-2q/T,x+2q/T]} \Phi \right)\notag
\end{align*}
where $I_q$ is defined in \eqref{eq:Iesti}, and we used that Condition \eqref{eq:i} applies, that is, $x-2q/T$ must exceed $x_1$.

Using \eqref{eq:i},  Ganelius' lemma, (contained in Lemma \ref{lemmegalten} above) implies $\|g\|_{\infty}\leq 16(a+\Phi(0))$.

With $q\geq \left(\pi/4-1/2\right)^{-1}$ (hence $\pi q/4-1\geq q/2$ and $q\geq 3$), one gets for $x\geq 4q/T$
\begin{align}\label{eq:gz}
|g(x)| & \leq  \left\{ (\lfloor 2q\rfloor +1) \left( a+\Phi(x -2q/T)\right) + \frac{16(a+\Phi(0))}{\pi q/4-1} \right\} \notag
\\ & \leq   \left(2+\frac{1}{3}\right)q \left( a+\Phi(x /2)\right) + \frac{32(a+\Phi(0))}{q}.
 \end{align}

Now,
$$\sqrt{\frac{32(a+\Phi(0))}{\left(2+\frac{1}{3}\right)\left(a+\Phi(x/2)\right)}} \geq 4\sqrt{\frac{6}{7}}\geq \left(\pi/4-1/2\right)^{-1},$$
so we can choose $q=\sqrt{\frac{32(a+\Phi(0))}{\left(2+{1}/{3}\right)\left(a+\Phi(x/2)\right)}} $ in \eqref{eq:gz}. With such a $q$ and for $x>\frac{16}{T}\sqrt{1+\Phi(0)/a}$, we also have
$x\geq 4q/T$, and in particular $x-2q/T\geq x/2\geq x_1$, thus we obtain
\begin{align*}
|g(x)|&\leq  8\sqrt{\frac{14}{3}}\sqrt{(a+\Phi(0))\left(a+\Phi(x/2)\right)}.
\end{align*}
\end{proof}

\begin{lemme}[Tenenbaum type Lemma]\label{l:T}
Let $\Phi$ be a nonnegative even function on $\reel$, nonincreasing on the half-line $[0,\infty)$. Let $b>0$ and $T > 0$ be real numbers and $g: \reel \to \reel$ be a bounded measurable function satisfying
\begin{equation}\label{eq:iii}
\sup_{0<y<1/T} \left\{ g(x+y)-g(x) \right\} \leq b + \Phi(x) \quad \left( x\geq x_1:=\frac{8}{T}\right).%%%%%%%\sqrt{1+\frac{\Phi(0)}{b}}
\end{equation}
Then, for all $x>x_0(T,b)=(16/T)\sqrt{1+\Phi(0)/b}$, we have
\begin{align}\label{eq:Ten}
|g(x)| \leq 20&\Bigg(\sqrt{\Phi(0)\Phi(x/2)}+ \sqrt{\Phi(0)}\sqrt{b+\frac{1}{\pi}\int_{-T}^{T}|\widehat{g}(\tau)|d\tau}\nonumber\\
& +b+\frac{1}{\pi}\int_{-T}^{T}|\widehat{g}(\tau)|d\tau\Bigg).
\end{align}
\end{lemme}

\begin{proof} We follow Tenenbaum's proof of Lemma \ref{lemmegalten}, i.e. Théorème 7.15 on page 328 of \cite{T2}, replacing the Lemma of Ganelius (see \cite[Th\'eor\`eme 7.14]{T2}) by Lemma \ref{l:G} in the argument. So for $\ve>0$ we define on $\reel$ the function
$$
\al(t):=\frac{2}{\pi\ve t^2} \sin\left(\frac{\ve t}{2}\right)\sin\left(\frac{2T+\ve}{2}t\right)
$$
having the Fourier transform

$$
\widehat{\al}(\tau)=\begin{cases} 1 \qquad \qquad\qquad &
\textrm{if}~~|\tau|\leq T \\
\frac{T+\ve-|\tau|}{\ve} & \textrm{if}~~T< |\tau|\leq T+\ve \\
0 & \textrm{if}~~ |\tau| >T+\ve
\end{cases}
$$
and the functions $f:=g*\al$ and $h:=g-f$. For any real number $x$, we obviously have $|g(x)|\leq |h(x)| + \|f\|_{\infty}$, furthermore, thanks to the Condition \eqref{eq:iii}, we also have
\begin{align}\label{hch}
\sup_{0<y<1/T} \{h(x+y) -h(x)\} &\leq 2\|f\|_{\infty} +
\sup_{0<y<1/T} \left\{g(x+y) -g(x)\right\} \\ & \leq
\frac{1}{\pi} \int_{-T-\ve}^{T+\ve} |\widehat{g}(\tau)|d\tau +
 b + \Phi(x)  ,\notag
\end{align}
where we used $\|{f}\|_\infty\leq \|\widehat{f}\|_1/(2\pi)$ and $|\widehat{f}|=|\widehat{\al}\widehat{g}|$ for the last inequality.\\
On the other hand,
$\widehat{h}=\widehat{g}-\widehat{f}=\widehat{g}(1-\widehat{\al})=0$
on the interval $[-T,T]$, whence Lemma \ref{l:G} can be applied
with $a:=b+Q$, $Q:=\frac{1}{\pi}\int_{-T-\ve}^{T+\ve}
|\widehat{g}|$. Noting that $b\leq a$ implies $x_0(T,b)\geq
x_0(T,a)$ while $x_1(a,T)>8/T$, this gives for $x\geq x_0(T,b)$,
\begin{align*}
|h(x)|&\leq 8\sqrt{\frac{14}{3}}\sqrt{(b+Q+\Phi(0))(b+Q+\Phi(x/2))}.
\end{align*}
Adding $~|f(x)|\leq \|f\|_{\infty}\leq\frac{1}{2}Q~$ yields the inequality
\begin{align*}
|g(x)| & \leq 8\sqrt{\frac{14}{3}}\sqrt{(b+Q+\Phi(0))(b+Q+\Phi(x/2))}+\frac{1}{2} Q\\
&\leq 20\sqrt{(b+Q+\Phi(0))(b+Q+\Phi(x/2))}.
\end{align*}
After letting $\ve$ tend to $0$, the estimate \eqref{eq:Ten} results by comparing squares of the expressions on the right there and here. \end{proof}

%%%%%%%%%%%%%%%%%%%%%%%%%%%%%%%%%%%%%%%%%%%%%%%%%%%%%%%%%%%%%%%%%%%%%%%%%%%%%%%

\section{Effective Wiener-Ikehara Theorem}\label{sec:multiplepoles}

To get a theorem with an effective error term, we shall need a strong restriction on both the decrease of the function $A$ and the regularity of the Mellin-Stieltjes transform on the border line.

We first recall the effective Wiener-Ikehara theorem in Tenenbaum's version \cite[p.326]{T2}.
\begin{theoreme}[Effective Wiener-Ikehara]\label{ikin}
Let $A$ be a nondecreasing function such that the integral
$\mathcal{A}(s)=\int_0^\infty e^{-st}dA(t)$ is convergent for $\sigma>1$.
Furthermore, let $\alpha \geq 0$ and $\omega>-1$ be constants and let us define the functions
\begin{equation}\label{eq:Gdefwomega}
G(s):=\frac{1}{s+1}\mathcal{A}(s+1)-\frac{\alpha}{s^{\omega+1}}
\end{equation}
and correspondingly
\begin{equation}\label{eq:etadef}
\eta_G(\sigma,T):=\int_{-T}^{T}\left\vert G(2\sigma+i\tau)-G(\sigma+i\tau)\right\vert d\tau.
\end{equation}
If we assume that the above functions \eqref{eq:Gdefwomega} and \eqref{eq:etadef} satisfy for any fixed $T>0$ that $\sigma^\omega \eta_G(\sigma,T)=o(1)$ as $\sigma\rightarrow 0^{+}$, then we have (with an explicit $O$-constant depending only on $\alpha$ and $\omega$)
$$
A(t)= e^{t} t^\omega \left(\frac{\alpha}{\Gamma(\omega+1)}+O(r(t))\right) \qquad (t \geq 1),
$$
where
\begin{equation}\label{eq:defR}\notag
r(t)=\inf_{T\geq 64}\left(\frac{1}{T}+\frac{1}{t^\omega}\eta_G\left(\frac{1}{t},T\right)+\frac{1}{(Tt)^{\omega+1}}\right).
\end{equation}
\end{theoreme}

\begin{remark}
The hypothesis $\sigma^\omega \eta_G(\sigma,T)=o(1)$ in Theorem \ref{ikin} ensures that $r(t)$ be $o(1)$ when $t \to\infty$. This hypothesis can be weakened or withdrawn but the theorem may thus lead to a weaker or trivial upper bound.
\end{remark}

Note that in case $\mathcal{A}$ has a meromorphic continuation on the boundary line, the announced upper bound is nontrivial, since we have the following fact, certainly well-known and folklore, but difficult to trace back as to its origin (c.f. \cite[Exc. 224]{T2} for one direction of the statement).

\begin{proposition}\label{prop:ordermequiv} Let $T$ be a non negative real number and $G$ be a meromorphic function in some open domain containing the closed interval $[-iT,iT]$ with $-iT$ and $iT$ regular points of $G$. With $\eta_G(\sigma,T)$ as above in \eqref{eq:etadef}, the meromorphic function $G$ has poles of order at most $m$ in $(-iT,iT)$, if and only if $\sigma^{m-1}\eta_G(\sigma,T)=O(1)$ as $\sigma\to 0+$. In particular $G$ has no singularities on the interval $[-iT,iT]$ if and only if
$\lim\limits_{\sigma\to 0+} \eta_G(\sigma,T)=0$.\\
Furthermore, if $H=G-P$ where $P$ is the principal part of the function $G$ on $[-iT,iT]$, then we can write
$$
\eta_G(\sigma,T)\leq \sum_{k=0}^{m-1}a_k\sigma^{-k}+\eta_H(\sigma,T),
$$
where $\lim\limits_{\sigma\to 0+} \eta_H(\sigma,T)=0$ and the coefficients $a_k$ can be explicitly computed in terms of the coefficients of the principal part $P$.
\end{proposition}

\begin{proof} First let us assume that the meromorphic function $G$ has some pole  $q=ib$, with $-T<b<T$, of order $m\geq 1$. We are to prove that $\sigma^{m-1}\eta_G(\sigma,T)$ is bounded from below by a positive constant.\\
There exists $\alpha\in\CC$ such that for $K>0$ and $\ve>0$, we can find $\delta=\delta(K,\ve)>0$, $p\in\comp^*$ and a function $h$ such that  $|b\pm K\delta|<T$,  $h$ is analytic on $\mathcal{R}:=[0,2\delta]\times [(b-K\delta)i ,(b+K\delta)i ]$ and for any  $s\in\mathcal{R}$, we have $G(s)=\al(s-q)^{-m}(1+h(s))$ and $|h(s)|\leq \ve$.\\
For $0<\sigma<\delta$, we therefore have
\begin{align*}
\eta_G(\sigma,T)&>\int_{b-K\sigma}^{b+K\sigma} \left|G(2\sigma+it)-G(\sigma+it)\right|dt
\\
& \geq
|\al|\int_{-K\sigma}^{K\sigma}(1-\ve)|(\sigma+it)^{-m}|-(1+\ve)|(2\sigma+it)^{-m}|dt
\\
& \geq
2|\al|\left((1-2\ve)W_m(\si,K\si)-W_m(2\si,K\si)\right)
\\
&\geq \frac{2|\al|}{\sigma^{m-1}} \left\{(1-2\ve)\int_{0}^{K}
\frac{du}{(1+u^2)^{m/2}} - 2^{-m+1} \int_{0}^{K/2} \frac{du}{(1+u^2)^{m/2}}
\right\}.
\end{align*}
For $m\geq 2 $ and $\ve\leq 1/4$, we thus have
$$
\eta_G(\sigma,T)\sigma^{m-1}>|\al|\int_{K/2}^{K}(1+u^2)^{-m/2} du,
$$
where the expression on the right hand side is a positive constant, depending only on $\al$, $m$ and $K$.\\
For $m=1$, we have
\begin{align*}
\eta_G(\sigma,T)&\geq 2|\al|
\left\{\int_{K/2}^{K} (1+u^2)^{-1/2} du- 2\ve \int_{0}^{K} (1+u^2)^{-1/2} du \right\}
\\
&\geq 2|\al| \left\{\frac{K}{2\sqrt{1+K^2}}-2\ve(1+\ln{K})\right\}.
\end{align*}
This lower bound is a constant, depending only on $\al, K$ and $\ve$ and strictly positive whenever $K\ge 3$ and $\ve<1/(8\sqrt{2}\ln{K})$.
%We have thus proved that $\eta_G(\sigma,T)$ does not converge to $0$ when $\sigma$ tends to $0^+$ if $G$ has some singularity on the interval $(-iT,iT)$.

\bigskip

If we take now a function $H$, analytic on $[-iT,iT]$, then we can choose $\delta>0$ such that $H$ is analytic on $\mathcal{D}:=[0,2\delta] \times [-iT,iT]$
%(meaning regularity also in some open neighborhood of $\mathcal{D}$)
and $K>0$ such that $|H'(s)|\leq K$ for $s\in\mathcal{D}$. For any $\sigma<\delta$ we get
\begin{align*}
\eta_H(\sigma,T)& =\int_{-T}^{T} \left| \int_{\sigma}^{2\sigma} H'(\xi+it) d\xi \right| dt \leq \int_{-T}^{T} \int_{\sigma}^{2\sigma} \left|H'(\xi+it)\right|d\xi dt \leq 2T \sigma K,
\end{align*}
which obviously tends to zero together with $\sigma$. \\
%Moreover, let now $\eta_H(\sigma,T)$ be the function like in \eqref{eq:etadef}, but $H$ replacing $G$. As $H$ is analytic on $[-iT.iT]$, the above argument yields that $\lim\limits_{\sigma\to 0+} \eta_H(\sigma,T)=0$.
It remains now to calculate the upper estimation of $\eta_Q(\sigma,T)$ for a general singular term $Q(s)=a (s-ib)^{-j}$, where $a\in\CC$, $|b|<T$ and $j\in\NN^*$, in the principal part $P(s)$ of $G(s)$.  We can write
\begin{align*}
\eta_Q(\sigma,T)&\leq \int_{-\infty}^{\infty} \left| \frac{a}{(2\sigma+i(t-b))^j} - \frac{a}{(\sigma+i(t-b))^j} \right|dt
\\
%&= |a| \int_{-\infty}^{\infty} \left|\frac{\sigma\{(2\sigma+i(t-b))^{j-1}+\dots +(\sigma+i(t-b))^{j-1}\}} {(2\sigma+i(t-b))^j(\sigma+i(t-b))^j} \right|dt
%\\
& \leq |a|  j \sigma \int_{-\infty}^{\infty} \frac{dt}{|\sigma+it|^{j+1}}=
2  |a| j\sigma W_{j+1}(\si,\infty)
\leq \frac{2 |a| (j+1)}{\sigma^{j-1}},
\end{align*}
according to \eqref{Wsi} and \eqref{estwallis2}.\\
Summing the contributions of the poles, we get an explicit upper bound of the desired form for $\eta_P(\sigma, T)$ only depending on the coefficients of the principal part $P$ of $G$.
\end{proof}

Theorem \ref{ikin} apply immediately to Dirichlet series with a pole at $s=1$.
\begin{theoreme}[Meromorphic Wiener-Ikehara]\label{ikinmer}
Let $(a_n)_n$ be a sequence of positive real numbers such that the Dirichlet series $\mathcal{A}(s)=\sum_{n\geq 1}a_nn^{-s}$ is convergent for $\sigma>1$ and has a meromorphic continuation on $\sigma\geq 1$ with one single pole at $s=1$, with order $m\in\NN$ and $\al:=\lim_{s\rightarrow 0+}s^{m}\mathcal{A}(s+1)$. With the functions $G(s)$ and $\eta_G(\sigma,T)$ defined in \eqref{eq:Gdefwomega} and \eqref{eq:etadef} above, we have
$$
A(x):=\sum_{n\leq x}a_n=x(\log{x})^{m-1}\left(\frac{\al}{(m-1) !}+O(\rho(x))\right)
$$
with %%%%%%%%%%%%%%%%$\rho_T(x):=\inf_{T\geq 64}\rho_T(x) $, where for any $T>0$ we define
\begin{equation}\label{eq:rhoxdef}\notag
\rho(x):=\inf_{T\geq 64}\left(\frac{1}{T}+\frac{1}{\log^{m-1}{x}}~ \eta_G\left(\frac{1}{\log{x}},T\right) \right).
\end{equation}
\end{theoreme}

We get a result similar to Theorem \ref{ikinmer} also for complex sequences majorized by "good" sequences.
\begin{corollary}\label{cormajdir}
Let $(a_n)$ be a sequence of complex numbers and $(b_n)$ be a sequence of positive numbers such that the Dirichlet series $\mathcal{A}(s)=\sum_{n\geq 1}a_nn^{-s}$ and $\mathcal{B}(s)=\sum_{n\geq 1}b_nn^{-s}$ are absolutely convergent for $\sigma>1$ and have meromorphic continuations on $\sigma\geq 1$ with at most one single pole, simple, at $s=1$,
with residue $\alpha$ and $\beta$ respectively. We assume further that $|a_n|\leq b_n$ and we define the functions
$$
G_a(s)=\frac{1}{s+1}\mathcal{A}(s+1)-\frac{\alpha}{s}, \qquad
G_b(s)=\frac{1}{s+1}\mathcal{B}(s+1)-\frac{\beta}{s}
$$
and the associated $\eta$ functions $\eta_a$ and $\eta_b$ the way we have done in \eqref{eq:Gdefwomega}-\eqref{eq:etadef}.\\
Then we have
$$A(x):=\sum_{n\leq x}a_n=\alpha x+O( x\rho(x))$$
with
$$\rho(x)=\inf_{T\geq 64}\left(\frac{1}{T}+\eta_b\left(\frac{1}{\log{x}},T\right)+\eta_a\left(\frac{1}{\log{x}},T\right)\right).$$
\end{corollary}
To prove this corollary, it is enough to consider $a_n^{(1)}:=\Re{a_n}+b_n$ and  $a_n^{(2)}:=\Im{a_n}+b_n$. This had already been noticed for the non effective Wiener-Ikehara Theorem in \cite[p.7]{M} by M.R and V.K Murty and can be found also in the context of the Karamata Theorem in \cite[Th. 7.7, page 318]{T2}.\\
However, there is no satisfactory description of the class of Dirichlet series with their coefficients admitting a positive majorization by some sequence subject to the Wiener-Ikehara Theorem. Moreover, even if such a majorization is found, from the analytic behavior of the majorant, no consequence as to the analytic behavior of the original Dirichlet series can be drawn. For example, the sequence $a_n:=2\cos(\log n) \Lambda(n)$ (which is the coefficient sequence of the Dirichlet series $\zeta'/\zeta (s+i) + \zeta'/\zeta (s-i)$), is clearly majorized by $2 \Lambda(n)$, the coefficient sequence of $2 \zeta'/\zeta$, which has no singularities on the 1-line (apart from the simple pole at $s=1$), but the original Dirichlet series admits two other singularities at $1\pm i$, too.
Therefore, it seems that descriptive formulation of such majorization results is not convenient, and that the case dependent study of the situation (singularities, which
majorization to use, etc.) is more appropriate.

\bigskip
We can now state our first main result which gives an effective error term for the asymptotic estimate of a moderately decreasing function.
\begin{theoreme}\label{th:kle}
Let $A$ be a real function of locally bounded variation,
vanishing on $(-\infty,1)$, such that the Mellin transform
$\mathcal{A}(s):=\int_{1}^{\infty} x^{-s} dA(x)$ is convergent for
$\sigma>1$. We define the function
\begin{equation}\label{eq:Gdefwresidues}
G(s)=\frac{1}{s+1}\mathcal{A}(s+1) -\sum_{k=0}^{n}\sum_{l=1}^{m_k}\left(\frac{c_{k,l}}{(s+ib_k)^{\omega_{k,l}+1}}+\frac{\overline{c_{k,l}}}{(s-ib_k)^{\omega_{k,l}+1} }\right),
\end{equation}
where $n$ and $m_k$ are positive integers, $c_{k,l}$ are complex numbers, $\omega_{k,l}$ and $b_k$ are real numbers satisfying $b_0=0$ and $b_k>0$ if $k\geq 1$ and $\omega_{k,l}>-1$. Furthermore, for $T>0$ and $\sigma>0$, we also define the function
\begin{equation}\label{eq:etadefwres}
\eta(\sigma,T)=\int_{-T}^{T}\left\vert G(2\sigma+i\tau)-G(\sigma+i\tau)\right\vert d\tau.
\end{equation}

Assume that the function $A$ is moderately decreasing on $[0,+\infty)$ in the sense of Definition \ref{def:moderatedec}.

Then for $x\geq  x_0(T):=\max(e^{16/15}; e^{32/T})$, we have
\begin{equation}\label{eq:Aasymptotic}
\left\vert A(x)
%-\sum_{l=1}^{m_0}\frac{c_{0,l}(\log{x})^{\omega_{0,l}}}{\Gamma(\omega_{0,l}+1)}
-x\sum_{k=0}^{n}\sum_{l=1}^{m_k}\frac{(\log{x})^{\omega_{k,l}}}{\Gamma(\omega_{k,l}+1)}2\Re\left(c_{k,l}e^{-ib_k\log{x}}\right)
\right\vert
\ll  x\rho_T(x)
\end{equation}
where the implicit constant depends only on $B_1$, $B_2$, $c_{k,l}$, $\Om:=\max_{k,l} \omega_{k,l}$, $b_k$ and $\|A_-\|_{[1,e]}$, and where  with $\omega:=\min_{k,l}\omega_{k,l}$ and $\Om:=\max_{k,l} \omega_{k,l}$,
\begin{align}\label{eq:rhodef}
\rho_T(x) =\sqrt{\varphi(\sqrt{x})}+\max_{\kappa=1/2 ,1} & \Bigg\{ %%%%%%%\frac{e^{10/T}}{T}  +
\left(\frac{e^{10/T}}{T}+1\right)\eta\left(\frac{1}{\log{x}},T\right)+
\\ & \frac{e^{10/T}}{T} \log^{\Om_{+}} x +\frac{1}{\log x} \left(\frac{1}{T^{\Om+1}}+\frac{1}{T^{\omega+1}}\right)\Bigg\}^\kappa , \notag
\end{align}
with $\varphi$ being the function introduced in \eqref{eq:moderatedec} in Definition \ref{def:moderatedec}.
\end{theoreme}
\begin{remark}
The upper bound we get only has a chance to be small when $T$ is large. In case $T\geq T_0>0$
we get an upper bound of the form
\begin{align*}
\rho_T(x) &=\sqrt{\varphi(\sqrt{x})}+\max_{\kappa=1/2 ,1} \left\{\eta\left(\frac{1}{\log{x}},T\right)+ \frac{1}{T^{\omega+1}\log{x}} + \frac{\log^{\Om_+} x}{T}\right\}^\kappa,
\end{align*}
where the implicit constant depends on the previous parameters and on $T_0$.
\end{remark}
\begin{remark}
Korevaar gives some tauberian theorem for a similar type of functions in \cite{Ko}, p.381-382 but his result seems less general than ours.
\end{remark}

\begin{proof}
Our proof combines the arguments of the proof of Proposition III.10.2 in \cite{Ko} and of the proof of Th\'eor\`eme II.7.13 in \cite{T2} with our Lemma \ref{l:T}.

We define for  $t\in\reel$ and $\si\in(0,1)$, the functions
$$g_\sigma(t):=h_\si(t)-h_{2\si}(t)=h_\sigma(t)(1-e^{-\sigma t})=A(e^t)e^{-(1+\sigma)t}(1-e^{-\sigma t})$$
and with $\beta(\omega,t)$ defined in \eqref{eq:betadef} also
\begin{align*}
L_\sigma(t):=g_\sigma(t)
%-& \sum_{l=1}^{m_0}c_{0,l}\si^{-\omega_{0,l}}B_{0,l}(\si t)
&-2\sum_{k=0}^{n}\sum_{l=1}^{m_k}\Re\left(c_{k,l}e^{-ib_kt}\right)\si^{-\omega_{k,l}}\beta(\omega_{k,l},\sigma t).
\end{align*}
%%%%%%%%with
%%%%%%%%$$
%%%%%%%%\beta(\omega,t):=\left\{\begin{array}{ll} \frac{t^{\omega}}{\Gamma(\omega+1)}e^{-t}(1-e^{-t})&(t>0)\\0&(t\leq0).\end{array}\right.
%%%%%%%%$$
We shall use Lemma \ref{l:T} to give an upper bound for the function $L_\sigma$.
Let  $T>0$. We shall need an upper bound for the integral
$$\int_{-T}^{T}|\widehat{L_\sigma}(\tau)|d\tau \quad \mbox{ hence for the integral }\quad \int_{-T}^{T}|\widehat{g_\sigma}(\tau)|d\tau$$
and a lower bound for the difference $L_\sigma(u+v)-L_\si(u)$, hence of  $g_\sigma(u+v)-g_\si(u)$, for arbitrary values of $\sigma\in(0,1)$, $u\in\reel$ and $0\leq v\leq 1/T$.

\bigskip
The Fourier transform of $g_\sigma$ is
\begin{align*}%%%\label{eq:ghatGsum}
\widehat{g}_\sigma(\tau)& =\frac{\mathcal{A}(\sigma+i\tau+1)}{\sigma+i\tau+1}
- \frac{\mathcal{A}(2\sigma+i\tau+1)}{2\sigma+i\tau+1}
=G(\sigma+i\tau)-G(2\sigma+i\tau) \notag
\\&+\sum_{k=0}^{n}\sum_{l=1}^{m_k}\left(c_{k,l}\left(\frac{1}{(\si+i(\tau+b_k))^{\omega_{k,l}+1}}
-\frac{1}{(2\si+i(\tau+b_k))^{\omega_{k,l}+1}}\right)\right.\notag 
\\&\quad \quad\quad\quad\left.+~\overline{c_{k,l}}\left(\frac{1}{(\si+i(\tau-b_k)^{\omega_{k,l}+1} }- \frac{1}{(2\si+i(\tau-b_k)^{\omega_{k,l}+1} }\right)\right),\notag
\end{align*}
hence those of $L_\sigma$ is
$\widehat{L_\sigma}(\tau)=G(\sigma+i\tau)-G(2\sigma+i\tau).$
Therefore
\begin{align}\label{Lsigmachapeau}
\int_{-T}^{T}|\widehat{L_\sigma}(\tau)|d\tau =\eta(\sigma,T).
\end{align}
For any $b \in \RR$ and $-1<\om\leq 0$ we can infer
$$
\left\vert\frac{1}{(\si+i(\tau\pm b))^{\omega+1}}-\frac{1}{(2\si+i(\tau \pm b))^{\omega+1}}\right\vert\leq \frac{(\omega+1)\sigma}{|\sigma+i(\tau\pm b)|^{\omega+2}},
$$
expressing the difference by an integral, and so \eqref{defwallis}, \eqref{Wsi} and \eqref{estwallis2} leads to
$$
\int_{-T}^{T} \left\vert\frac{1}{(\si+i(\tau\pm b))^{\omega+1}}-\frac{1}{(2\si+i(\tau \pm b))^{\omega+1}} \right\vert d\tau \leq 2 (\omega+2)\si^{-\omega}\leq 6\si^{-\omega},
$$
while for any $b\in \RR$ and $\om>0$ we integrate the terms separately and apply \eqref{defwallis}, \eqref{Wsi} and \eqref{estwallis2} with $m=\om+1>1$ giving the estimate $\leq 2 (1+2^{-\om}) \frac{\omega+1}{\omega} \si^{-\omega} \leq 6\si^{-\omega}$ for the integral on the left hand side.
%%%%%%$$
%%%%%%\int_{-T}^{T} \left\vert\frac{1}{(\si+i(\tau\pm b))^{\omega+1}}-\frac{1}{(2\si+i(\tau \pm b))^{\omega+1}} \right\vert d\tau \leq 2 (1+2^{-\om}) \frac{\omega+1}{\omega} \si^{-\omega} \leq 6\si^{-\omega}.
%%%%%%$$
%%%%%%%%%%For $-1<\om_{k,l}\leq 1$, we use
%%%%%%%%%%$$
%%%%%%%%%%\left\vert\frac{1}{(\si+i(\tau\pm b_k))^{\omega_{k,l}+1}}-\frac{1}{(2\si+i(\tau \pm b_k))^{\omega_{k,l}+1}}\right\vert\leq \frac{(\omega_{k,l}+1)\sigma}{|\sigma+i(\tau\pm b_k)|^{\omega_{k,l}+2}},
%%%%%%%%%%$$
%%%%%%%%%%and \eqref{defwallis}, \eqref{Wsi} and \eqref{estwallis2} to get
%%%%%%%%%%\begin{align*}
%%%%%%%%%%\int_{-T}^{T} \left\vert\frac{1}{(\si+i(\tau\pm b_k))^{\omega_{k,l}+1}}-\frac{1}{(2\si+i(\tau \pm b_k))^{\omega_{k,l}+1}} \right\vert d\tau &\leq 2 (\omega_{k,l}+2)\si^{-\omega_{k,l}}\\&\leq 6\si^{-\omega_{k,l}},
%%%%%%%%%%\end{align*}
%%%%%%%%%%while for $\om_{k,l}>0$ we integrate the terms separately and get
%%%%%%%%%%\begin{align*}
%%%%%%%%%%\int_{-T}^{T} \left\vert\frac{1}{(\si+i(\tau\pm b_k))^{\omega_{k,l}+1}}-\frac{1}{(2\si+i(\tau \pm b_k))^{\omega_{k,l}+1}} \right\vert d\tau &\leq 2 (1+2^{-\om_{k,l}}) \frac{\omega_{k,l}+1}{\omega_{k,l}} \si^{-\omega_{k,l}}\\&\leq 6\si^{-\omega_{k,l}}.
%%%%%%%%%%\end{align*}
So in all cases we have obtained an estimate by $6\si^{-\om}$. Integrating the above expression for $\widehat{g}_\sigma(\tau)$ and using this estimate with $\om=\om_{k,l}$ and $b=b_k$ leads to
\begin{align}\label{gsigmachapeau}
\int_{-T}^{T}|\widehat{g_\sigma}(\tau)|d\tau\leq &\eta(\sigma,T)
+ \sum_{k=0}^{n}\sum_{l=1}^{m_k}|c_{k,l}| ~12 ~ \si^{-\omega_{k,l}}.
\end{align}

Next we give a lower bound for the difference $L_\sigma(u+v)-L_\si(u)$, hence of  $g_\sigma(u+v)-g_\si(u)$, for arbitrary values of $\sigma\in(0,1)$, $u\in\reel$ and $0\leq v\leq 1/T$.\\
Since $A$ vanishes on $(-\infty,1)$, we assume $u+v\geq 0$. For $\sigma\in(0,1)$, $u\in\reel$ and $0\leq v\leq 1/T$, we have
\begin{align*}
g_\sigma(u+v)-g_\sigma(u) =F_1(u,v)+g_\sigma(u)F_2(u,v)
\end{align*}
with
\begin{align*}
F_1(u,v)&=e^{-(1+\sigma)(u+v)}\left(1-e^{-\sigma(u+v)}\right)\left( A\left(e^{u+v}\right)-A\left(e^u\right)\right)\\
F_2(u,v)&=e^{-(1+\sigma)v} \frac{e^{-\sigma u}-e^{-\sigma(u+v)}}{1-e^{-\sigma u}}+e^{-(1+\sigma)v}-1.
\end{align*}

\begin{itemize}
\item  First, we give a lower bound for $g_\si(u)F_2(u,v)$.
\begin{itemize}
\item For $u<0$, $g_\si(u)F_2(u,v)$ obviously vanishes together with $A(e^u)$ and $g_\si(u)$, thus $g_\sigma(u)F_2(u,v)=0$.
\item Note that in case $u=0$, $F_2(u,v)$ is not defined but $$g_\sigma(u)F_2(u,v)=A(1)e^{-(1+\sigma)v} \left(1-e^{-\sigma v}\right)\geq -\si v\|A_{-}\|_{[1,e]}.$$
\item For any $\sigma\in(0,1)$, $u>0$ and $v\geq 0$, we have
$$F_2(u,v)=e^{-(1+\sigma)v} \frac{1-e^{-\sigma v}}{e^{\sigma u}-1}+e^{-(1+\sigma)v}-1,$$
 thus
$$
-2v\leq -(1+\si) v\leq -(1-e^{-(1+\si) v})\leq F_2(u,v)\leq \frac{1-e^{-\si v}}{e^{\si u}-1}\leq \frac{v}{u},
$$
and therefore $F_2(u,v)g_\sigma(u)>-2v\|g_\sigma\|_\infty$ if $A(e^u)\geq 0$.\\
If $A(e^u)< 0$ and $u\geq 1$, we use $F_2(u,v)\leq v$, thus $F_2(u,v)g_\sigma(u)>-v\|g_\sigma\|_\infty$.\\
In case $0<u<1$ and $A(e^u)< 0$, we write
$$F_2(u,v)g_\si(u) = A(e^u) (f(u+v)-f(u))\geq-\{A(e^u)\}_{-} |f(u+v)-f(u)|$$
with $f(u)=e^{-(1+\si)u}(1-e^{-\si u})$. Since
$\sup_{u\geq 0}|f'(u)|\leq 1+\si\leq 2$, we get
$$F_2(u,v)g_\si(u)\geq -2v\{A(e^u)\}_{-}\geq -2v\|A_{-}\|_{[1,e]} .$$
\end{itemize}
Therefore, for $u\in \RR$, $\si\in (0,1)$ and $v\geq 0$, we always have%%%%%\rev{C'est bon!}
$$
F_2(u,v)g_\sigma(u)\ge -2v \left( \|A_{-}\|_{[1,e]}+\|g_\sigma\|_\infty\right).
$$
\item To deal with the first term $F_1(u,v)$, we use \eqref{eq:moderatedec}.
\begin{itemize}
\item We therefore get for $\sigma\in(0,1)$, $u\geq 0$ and $v\geq 0$
\begin{align*}
F_1(u,v)\geq -\{F_1(u,v)\}_{-} & =- (1-e^{-\sigma(u+v)}) \frac{\{A(e^{u+v})-A(e^u)\}_{-}}{e^{(1+\sigma)(u+v)}}  \\ & \geq -e^{-(u+v)}\{A(e^{u+v})-A(e^u)\}_{-}\\
&\geq -B_1 (1-e^{-v})-B_2\varphi(e^u)e^{-v}\\
&\geq -B_1 v-\Phi(u),
\end{align*}
with $\Phi(u)=B_2\varphi(e^u)$.
\item For $\sigma\in(0,1)$, $u< 0$ and $u+v\geq 0$, we use $A(e^{u+v})-A(e^u)=A(e^{u+v})=A(e^{u+v})-A(0)$ and get
\begin{align*}
F_1(u,v)& \geq - (1-e^{-\sigma(u+v)}) \frac{\{A(e^{u+v})\}_{-}}{e^{(1+\sigma)(u+v)}}  \\ & \geq -(1-e^{-\sigma(u+v)}) \left(e^{-(u+v)}\{A(e^{u+v})-A(0)\}_{-}\right)\\
&\geq -(1-e^{-\sigma(u+v)}) \left(B_1 +B_2e^{-(u+v)}\right)\\
&\geq -(B_1+B_2)(1-e^{-\sigma(u+v)})\\
&\geq -(B_1+B_2)\sigma(u+v)\\
&\geq  -(B_1+B_2)v
\end{align*}
since $0\leq u+v\leq v$ and $\si<1$.
\end{itemize}
\end{itemize}
For $\sigma\in(0,1)$, $u\in\RR$ and $v\geq 0$, the above lower estimates yield
\begin{equation}\label{eq:gdiff0}
g_\sigma(u+v)-g_\sigma(u) \geq -2v\|g_\sigma\|_\infty -B' v-\Phi(u),
\end{equation}
where $B'=B_1+B_2+2\|A_{-}\|_{[1, e]}$ and $\Phi$ is defined by $\Phi(u):=B_2\varphi(e^{|u|})$.

\bigskip

We shall now use Lemma \ref{lemf} to give an upper bound for $\| g_\si\|_\infty$.

Since $A$ is moderately decreasing, \eqref{minF} provides for any real number $u$ the estimate
$A(e^u)e^{-u}\geq -\lambda_3$ with $\lambda_3=B_1+B_2$. Consequently,
$g_\sigma(u)\geq -\lambda_3$.

Let $\lambda_2$ be a positive fixed constant and define as in the proof of Theorem  \ref{th:kleTfixe},
$$S_0(\lambda_2):=\{ u\in \RR ~:~ A(e^u) >\lambda_2 e^u \}.$$
Note that here, we do not  assume $u\geq 0$. Nevertheless, for $u<0$, we have $A(e^u)=0$ thus obviously $u\notin S_0(\lambda_2)$ and also $g_\si(u)=0$. Therefore, if $u\not\in S_0(\lambda_2)$, then $g_\si(u)\leq \lambda_2$ both for $u<0$ and for $u\geq 0$, too.

We now choose $\lambda_2=5\max(B_1,B_2)e^{10/T}$ as before (preceding \eqref{minhplus}). Then using the first inequality of \eqref{minhplus}, we get for $u\in S_0(\lambda_2)\setminus\{0\}$, $v\in[0,10/T]$ and $0<\sigma<\sigma_0:=\min\left(1,\frac{\log 2}{10}T\right)$,
\begin{align*}
g_\si(u+v)& =(1-e^{-\si(u+v)})h_\si(u+v) \geq (1-e^{-\si(u+v)}) \frac45e^{-(1+\si)10/T} h_\si(u) \\
&= \frac45e^{-\si10/T} e^{-10/T} \frac{1-e^{-\si(u+v)}}{1-e^{-\si u}}g_\sigma(u)
\geq \frac25e^{-10/T} g_\sigma(u).
\end{align*}
Note that the estimate remains valid when $u=0$.\\
Now we can apply (the first part of) Lemma \ref{lemf} similarly as in the proof of Theorem \ref{th:kleTfixe} and with $\lambda_1:=\dfrac25 e^{-10/T}$ and $\lambda_2, \lambda_3$ chosen above. Taking into account \eqref{gsigmachapeau} we infer
\begin{align*}
\|g_\sigma\|_\infty
&\leq e^{ 10/T}\left(5\lambda_3 +  0.55 \int_{-T}^{T} \left|\widehat{g_\sigma}(\tau)\right| d\tau\right)
\\ &\leq e^{ 10/T}\left(5\lambda_3 +  0.6 \eta(\sigma,T)+6.6\sum_{k=0}^{n}\sum_{l=1}^{m_k}|c_{k,l}| \si^{-\omega_{k,l}}\right).
\end{align*}
Thus with \eqref{eq:gdiff0}, we get for $0<\sigma<\si_0$, $u\in \RR$ and $0<v<1/T$
\begin{align}\label{eq:gdiff}
g_\sigma(u+v)-g_\sigma(u) &\geq
-\frac{e^{ 10/T}}{T}\left(13.2\sum_{k=0}^{n}\sum_{l=1}^{m_k}|c_{k,l}| \si^{-\omega_{k,l}}\right.\\
\notag&+10\lambda_3 +  1.2 \eta(\sigma,T)\bigg) -\frac{B'}{T}-\Phi(u).
\end{align}

\bigskip

We shall now give a lower bound for the difference $L_\sigma(u+v)-L_\si(u)$. For any $b, \gamma \in \RR$ and $-1<\om$ Lemma \ref{l:beta} yields with $y:=\si v \in [0,1/T]$
\begin{align*}
&\left\vert  \beta\left(\omega,\sigma(u+v)\right)\cos(b(u+v)+\gamma)-\beta(\omega,\sigma u)\cos(bu+\gamma)\right\vert
\\ & \leq \left\vert  \beta(\omega,\sigma(u+v))- \beta(\omega,\sigma u) \right\vert + |\cos(b(u+v)+\gamma)- \cos(b u+\gamma)|\cdot |\beta(\omega,\sigma u)| 
\\ & \qquad \qquad \qquad \qquad \qquad \qquad \qquad \leq \frac{2}{\sqrt{\pi}}\left(\frac{\sigma}{T}+ \left(\frac{\sigma}{T}\right)^{\omega+1}+\frac{|b|}{T}\right).
\end{align*}
Applying this with $\gamma:=\arg{c_{k,l}}$, $b:=b_k$ and $\om:=\om_{k,l}$ gives
\begin{align}\label{eq:Ldifference}
L_\sigma(u+v)-L_\sigma(u) &\geq  g_\sigma(u+v)-g_\si(u)
\\& - \frac{4}{\sqrt{\pi}}\sum_{k=0}^{n}\sum_{l=1}^{m_k}|c_{k,l}|\left(\frac{\sigma}{T}+
\left(\frac{\sigma}{T}\right)^{\omega_{k,l}+1}+\frac{|b_k|}{T}\right)\si^{-\omega_{k,l}}. \notag
\end{align}
Taking \eqref{eq:gdiff} into account, a calculation gives for all $u\in\reel$, $0\leq v\leq 1/T$ and $0<\sigma\leq \si_0$
\begin{align*}
L_\sigma(u+v)-L_\sigma(u)
& \geq   -\frac{C_3(\si,T)}{T}  -\Phi(u) -\frac{4}{\sqrt\pi}\si\sum_{k=0}^{n}\sum_{l=1}^{m_k}|c_{k,l}|\left(\frac{1}{T}\right)^{\omega_{k,l}+1}
\end{align*}
with
\begin{align*}
C_3(\sigma,T)&:=e^{10/T} (1.2\eta(\sigma,T)+10\lambda_3) +B' \\
&  +\sum_{k=0}^{n}\sum_{l=1}^{m_k}|c_{k,l}| (13.2e^{10/T}+\frac{4}{\sqrt\pi}(\si+|b_k|)) \si^{-\omega_{k,l}}.\\
\end{align*}

Finally, we apply Lemma \ref{l:T} to $-L_\sigma$. Since $C_3(\sigma,T)>10\lambda_3\geq 10B_2=10\Phi(0)$, the validity of the estimate \eqref{eq:Ten} can be assured for $t\geq \frac{16}{T}\sqrt{1+T/10}$, hence for $t\geq
t_0:=32\max(1/30;{1}/{T})>1/\si_0$. Therefore, Lemma \ref{l:T} and \eqref{Lsigmachapeau} gives for such $t$ that
\begin{align*}
|L_\sigma(t)|  \leq &
20 \left(\frac{C_3(\sigma,T)}{T}+\frac{4}{\sqrt\pi}\si \sum_{k=1}^{n}\sum_{l=1}^{m_k}|c_{k,l}|\left(\frac{1}{T}\right)^{\omega_{k,l}+1}+\frac{1}{\pi}\eta(\si,T)\right)\\
&+20\sqrt{B_2} \sqrt{\frac{C_3(\sigma,T)}{T}+\frac{4}{\sqrt\pi}\si \sum_{k=1}^{n}\sum_{l=1}^{m_k}|c_{k,l}|\left(\frac{1}{T}\right)^{\omega_{k,l}+1}+\frac{1}{\pi}\eta(\si,T)}\\
&+20\sqrt{B_2\Phi(t/2)}.
\end{align*}
We choose $\sigma=1/t$ and $x=e^t$ and assume $x\geq x_0(T):=e^{t_0}=\max(e^{16/15}; e^{32/T})$ (hence $\sigma<\si_0$) to obtain
$$
\left\vert A(x)-x\sum_{k=0}^{n}\sum_{l=1}^{m_k}2\Re\left(c_{k,l}e^{-ib_k\log{x}}\right)\frac{(\log{x})^{\omega_{k,l}}}{\Gamma(\omega_{k,l}+1)}\right\vert \leq x\rho_T(x)
$$
with
\begin{equation}\label{rhomoderate}
\rho_T(x)=\frac{20 e}{1-e^{-1}}\left\{B_2\sqrt{\varphi(\sqrt{x})}+R_T(x)+\sqrt{B_2 R_T(x)}\right\}
\end{equation}
where
\begin{align*}
R_T(x)=&\frac{C_4(T)}{T}  + C_5(T)\eta\left(\frac{1}{\log{x}},T\right)
\\&+\frac{1}{T}\sum_{k=0}^{n}\sum_{l=1}^{m_k}
\left(C_{k,l}^{(1)} (\log{x})^{\omega_{k,l}}+\frac{C_{k,l}^{(2)}}{\log x}\left( (\log{x})^{\omega_{k,l}}+\left(\frac{1}{T}\right)^{\omega_{k,l}}\right)\right),
\end{align*}
$$C_4(T)=10\lambda_3 e^{10/T}+B', \quad C_5(T)=\frac{1.2e^{10/T}}{T}+\frac{1}{\pi},$$
$$C_{k,l}^{(1)}(T)=|c_{k,l}|\left(13.2e^{10/T}+\frac{4}{\sqrt\pi}|b_k|\right), \quad C_{k,l}^{(2)}=\frac{4}{\sqrt\pi}|c_{k,l}|.$$
\end{proof}

\begin{remark}\label{rem:noroots}
Observe that in case $B_2=0$ can be taken, the relative losses of taking squareroots in \eqref{rhomoderate} vanish, and we get the sharper estimate with
\begin{equation*}%\label{rhoforincreasing}
\rho_T(x)=\left(\frac{e^{10/T}}{T}+1\right)\eta\left(\frac{1}{\log{x}},T\right)+
\frac{e^{10/T}}{T} \log^{\Om_{+}} x +\frac{1}{\log x} \left(\frac{1}{T^{\Om+1}}+\frac{1}{T^{\omega+1}}\right)
\end{equation*}
in place of \eqref{eq:rhodef} in the statement of the theorem.
\end{remark}

\begin{corollary}\label{cor:generalGerald} Under the same conditions then in Theorem \ref{th:kle}, but assuming (instead of the condition of $A(x)$ moderately decreasing) that $A(x)$ is nonnegative and nondecreasing, we obtain the bound \eqref{eq:Aasymptotic} with the error function
\begin{equation}\label{eq:increasingbound}
\rho_T(x)=\left(\frac{e^{10/T}}{T}+1\right)\eta\left(\frac{1}{\log{x}},T\right)+
\frac{e^{10/T}}{T} \log^{\Om} x +\frac{1}{\log x} \left(\frac{1}{T^{\Om+1}}+\frac{1}{T^{\omega+1}}\right).
\end{equation}
In particular, when $T\geq 1$, we find \eqref{eq:Aasymptotic} to hold with
\begin{equation}\label{eq:Geraldbound}
\rho_T(x)=\eta\left(\frac{1}{\log{x}},T\right)+\frac{\log^{\Om} x}{T} + \frac{1}{\log x ~ T^{1+\omega}}.
\end{equation}
\end{corollary}
\begin{proof} Compared to Remark \ref{rem:noroots} the only change we need to explain is the use of $\log^{\Om} x$ in place of $\log^{\Om_{+}} x$. But when $A$ is nondecreasing, then not only $B_2=0$, but also we have the condition \eqref{eq:moderatedec} even with $B_1=0$. Moreover, from nonnegativity of $A$ it also follows that $A_1:=\|A_{-}\|_{[1,e]}=0$, so also $B'=0$ and in the last formula of the proof of Theorem \ref{th:kle} giving $R_T(x)$, we find $C_4(T)=0$, too. Therefore, no term of the order $1/T$ appears, and instead of $1/T +\log^{\Om} x/T \asymp \log^{\Om_{+}} x /T$ we can as well write $\log^{\Om} x/T$.

When we even have $T\geq 1$, \eqref{eq:increasingbound} clearly entails \eqref{eq:Geraldbound}, too.
\end{proof}

\begin{remark}\label{rem:Gerald} Observe that the second part of Corollary \ref{cor:generalGerald} clearly covers the original result Theorem \ref{ikin} of Tenenbaum.
\end{remark}

A similar result can also be derived for functions satisfying the slowly decreasing condition, which is a more standard condition on the controlled decrease of functions.

\begin{theoreme}\label{th:kle1}
Let $A(x)$ be a real function of locally bounded variation, vanishing on $(-\infty,1)$, such that the Mellin transform $\mathcal{A}(s):=\int_{1}^{\infty} x^{-s} dA(x)$ converges for
$\sigma>1$. As in Theorem \ref{th:kle}, we define the functions $G(s)$ and $\eta_G(\si,T)$ according to \eqref{eq:Gdefwresidues} and \eqref{eq:etadefwres}, respectively, where $n$ and $m_k$ are positive integers, $c_{k,l}$ are complex numbers, $\omega_{k,l}$ and $b_k$ are real numbers satisfying $b_0=0$ and $b_k>0$ if $k\geq 1$, and $\omega_{k,l}>-1$, $T>0$ and $\sigma>0$.

Further we assume that $A(x)/x$ is slowly decreasing on $[1,\infty)$ in the sense of Definition \ref{def:Schmidt}, and that with $\Om:=\max_{k,l} \omega_{k,l}$ we have
\begin{equation}\label{eq:etacondifin}
\eta(\si,T)=O(\si^{-\Om_{+}}) \qquad (0<\si<1).
\end{equation}
Then for $x\geq  \max(e; e^{17/T},e^{17/\sqrt{T}})$ the error formula \eqref{eq:Aasymptotic} holds with
\begin{equation}\label{eq:rhoforslow}
\rho_T(x) =\sqrt{\Psi_{e^{1/T}}(1)\Psi_{e^{1/T}}(\sqrt{x})} +R_T(x)+\sqrt{\Psi_{e^{1/T}}(1)R_T(x)},
\end{equation}
where with $\omega:=\min_{k,l}\omega_{k,l}$
\begin{align}\label{eq:Rforslow}
R_T(x)= & \nu\left(\frac{A(x)}{x};e^{1/T}\right) + \left(1 +\frac{1}{T}\right)\eta\left(\frac{1}{\log x},T\right)\nonumber\\ &+ \frac{\log^{\Om_{+}} x +\log\log x}{T} + \frac{1}{T^{\omega+1}\log{x}}+\frac{1}{T^{\Om+1}\log{x}}
\end{align}
with $\overline{\nu}, \nu$ and $\Psi$ being the functions introduced in \eqref{eq:nubardef}, \eqref{eq:nudef} and \eqref{eq:nuPsilambda}, respectively, and where the implicit constants only depend on $\overline{\nu}(A(x)/x;e)$, $\|A\|_{[1,e]}$, $\AAA(2)$, the constant in \eqref{eq:etacondifin}, $c_{k,l}$ and $b_k$.
\end{theoreme}

\begin{proof}
Our proof is very similar to the proof of Theorem \ref{th:kle}, so we use the same notations.
As the control on the decrease was not used in that, we obtain the same upper bounds for the integrals
$$
\int_{-T}^{T}|\widehat{L_\sigma}(\tau)|d\tau \quad \mbox{ and }\quad \int_{-T}^{T}|\widehat{g_\sigma}(\tau)|d\tau,
$$
while we need to get a new lower bound for the difference $g_\sigma(u+v)-g_\si(u)$.
%%%%%, for arbitrary values of $\sigma\in(0,1)$, $u\in\reel$ and $0\leq v\leq 1/T$.

We now write for $\sigma\in(0,1)$, $u\in\RR$ and $0\leq v\leq 1/T$
\begin{align*}
g_\sigma(u+v)-g_\sigma(u) =F_3(u,v)+g_\sigma(u)F_4(u,v)
\end{align*}
with
\begin{align*}
F_3(u,v)&:=e^{-\sigma(u+v)}\left(1-e^{-\sigma(u+v)}\right)\left( \frac{A\left(e^{u+v}\right)}{e^{u+v}}-\frac{A\left(e^u\right)}{e^u}\right),\\
F_4(u,v)&:=e^{-\sigma v}\frac{1-e^{-\sigma(u+v)}}{1-e^{-\sigma u}}-1=e^{-\si v}-1+e^{-\si v}\frac{1-e^{-\si v}}{e^{\si u}-1}.
\end{align*}
For any $\sigma\in(0,1)$, $u > 0$ and $v\geq 0$ we immediately have
\begin{equation}\label{eq:F4}
-v \leq  -\si v\leq -(1-e^{-\si v})\leq F_4(u,v)\leq \frac{1-e^{-\si v}}{e^{\si u}-1}\leq \frac{v}{u}.
\end{equation}
For $\sigma\in(0,1)$, $u > 0$ and $0\leq v\leq 1/T$ we get from \eqref{eq:nuPsilambda}
\begin{align}\label{eq:F3}
F_3(u,v)& \geq -\{F_3(u,v)\}_{-} =-e^{-\sigma(u+v)}\left(1-e^{-\sigma(u+v)}\right)\left\{ \frac{A\left(e^{u+v}\right)}{e^{u+v}}-\frac{A\left(e^u\right)}{e^u}\right\}_{-}  \notag \\
& \geq -\left\{ \frac{A\left(e^{u+v}\right)}{e^{u+v}}-\frac{A\left(e^u\right)}{e^u}\right\}_{-}  \geq -\nu(e^{1/T})-\Psi_{e^{1/T}}(e^u).
\end{align}
For simplicity we will write $\nu(\cdot)$ for $\nu(A(x)/x,\cdot)$ and $\overline{\nu}(\cdot)$ for $\overline{\nu}(A(x)/x,\cdot)$ and we define the function $\Phi$ by $\Phi(u):= \Psi_{e^{1/T}}(e^{|u|})$. Then for any $\sigma\in(0,1)$, $u \geq 1$ and $0\leq v\leq 1/T$, \eqref{eq:F4} and \eqref{eq:F3} entail the lower estimate
\begin{equation}\label{eq:gdiff1}
g_\sigma(u+v)-g_\sigma(u) \geq -\nu(e^{1/T})-\Phi(u)-\frac{1}{T}\|g_\si\|_\infty.
\end{equation}
In case $\sigma\in(0,1)$, $0< u \leq 1$ and $0\leq v\leq 1/T$, we estimate $g_\si(u)F_4(u,v)$ from below together. The only interesting cases are when $g_\si(u)F_4(u,v)$ is negative, hence when either $g_\si(u)>0$ and $0 > F_4(u,v) > - v$, or when $g_\si(u)<0$ and $0<F_4(u,v)\leq  v/u$. In the first case the combined estimate is still as good as $g_\si(u)F_4(u,v) \geq  - v \|g_\si\|_\infty \geq  -\frac1T \|g_\si\|_\infty$, so \eqref{eq:gdiff1} remains valid. In the second case we make use that $A(x)/x$ is slowly decreasing, hence in particular also boundedly decreasing, and thus according to \eqref{minslowdec} it satisfies $A(e^u)/e^u > - M (u+1)$, where $M:=M(A,1)$. So in case $0<u\leq 1$ we are led to
$$
g_\si(u)=\frac{A(e^u)}{e^u} e^{-\si u} (1- e^{-\si u}) \geq - M(u+1) (1- e^{-\si u}) \geq -2M \si u
$$
and so when $g_\si(u)<0$ and $0<u\leq 1$, $0<\si<1$ and $0\leq v\leq 1/T$ then it holds
\begin{equation}\label{eq:gF4usmall}
g_\si(u)F_4(u,v) > -2 M \si u  \frac{v}{u} \geq - \frac{2M\si}{T}.
\end{equation}
From this and \eqref{eq:gdiff1} we thus conclude that for any $0<\si \leq 1$, $u>0$, $0<v\leq 1/T$, the estimate
\begin{equation}\label{eq:gdiff2}
g_\sigma(u+v)-g_\sigma(u) \geq -\nu(e^{1/T})-\Phi(u)- \frac{2M\si}{T} - \frac{1}{T}\|g_\si\|_\infty
\end{equation}
holds true. Furthermore, we found for $0<\si\leq 1$ and $u>0$ the lower estimate
\begin{align}\label{eq:gsibelow2}
g_\si(u) & \geq - M(u+1) e^{-\si u} (1- e^{-\si u}) \notag
\\ & \geq \begin{cases}-2M \si u &\textrm{if} ~ 0<u\leq 1, \\ -2M \max\limits_{u\geq 1} ue^{-\si u} \geq -\frac{2M}{e \si} &\textrm{if} ~ u>1 \end{cases} \notag \\
& \geq -\frac{2M}{\si} .
\end{align}
In case $-1 < \Om < 1$ and $u>e$, this lower bound has to be sharpened. We invoke Corollary \ref{c:localbddnessbyeta} with $m:=\Om+1$ under Condition {\it 1} listed in Theorem \ref{th:kleTfixe}. We get for $u\geq e$
\begin{equation}%\label{eq:improvedloweresti}
\frac{A(e^u)}{e^u} \geq \begin{cases} - K \log u &\textrm{if} \quad -1< \Om \leq 0 \\ - K u^{\Om} &\textrm{if} \qquad 0<\Om < 1 \end{cases},
\end{equation}
thus, if $0<\si<1$, we have for $u>e$
$$
g_\si(u)\geq \begin{cases} - K \log u ~e^{-\si u} \geq - {K} \log\left(2/\si\right) &\textrm{if}  \quad -1<\Om\leq 0 \\ - K u^{\Om} e^{-\si u}  \geq -K \left({\Om}/{e}\right)^\Om \si^{-\Om} &\textrm{if} \qquad 0<\Om < 1 \end{cases},
$$
while $g_\si(u) > -M(u+1) > -4M$ ($0 \leq u\leq e$) remain valid for all $\Om$, so we are led to the lower estimate
\begin{align}\label{eq:gsibelowmod}
g_\si(u) & \geq \begin{cases}- \max \left\{\frac{4 M}{\log 2} ,  {K} \right\} \log\left(2/\si\right) &\textrm{if}  \quad -1<\Om \leq 0  \\ - \max \left\{4M,K\left(\frac{\Om}{e}\right)^\Om \right\} / \si^\Om &\textrm{if} \qquad 0<\Om < 1 \end{cases} \notag
\\ & \geq  -(6M+K) \begin{cases} \log\left(2/\si\right) &\textrm{if} ~ \quad -1<\Om\leq 0 \\ \si^{-\Om} &\textrm{if} \qquad 0< \Om  <1 \end{cases},
\end{align}
valid for all $0<\si<1$ and $u\geq 0$, (so by $g_\si(u)=0$ $(u<0)$, for $u\in \RR$, too).

Writing
\begin{equation}\label{deflambda3}
\lambda_3:=  \begin{cases} {2M}/{\si} &\textrm{if}~ \quad \Om\geq 1\\(6M+K)\log\left(2/\si\right) &\textrm{if} ~ \quad -1<\Om\leq 0 \\ (6M+K)\si^{-\Om} &\textrm{if} \quad 0< \Om  <1 \end{cases}
\end{equation}
and taking \eqref{eq:gsibelow2} and \eqref{eq:gsibelowmod} into account,  we get  for all $0<\si<1$ and  $u\in \RR$
\begin{equation}\label{eq:gsibelow}
g_\si(u)\geq-\lambda_3.
\end{equation}

We choose now  $\lambda_2:=50\overline{\nu}(e^{1/T})+\|A\|_{[1,e]}$ and consider
$$
S_0(\lambda_2):=\{ u\in \RR ~:~ A(e^u) >\lambda_2 e^u \}.
$$
%Note that $u\in S_0(\lambda_2)$ can occur only for $u>e$, in view of the choice of $\lambda_2$. Moreover, as $g_\si(u)=A(e^u)/e^u e^{-\si u} (1-e^{-\si u})$, we certainly have $u\in S_0(\lambda_2)$ whenever $g_\si(u)>\lambda_2$.

Let $u\in\RR$ be such that $g_\si(u)>\lambda_2$. Then $u>e$ and $u\in S_0(\lambda_2)$ and for $0\leq v\leq 10/T$ we have
$$
\frac{A(e^{u+v})}{e^{u+v}}-\frac{A(e^u)}{e^u} \geq -\overline{\nu}(e^{10/T})\geq -10\overline{\nu}(e^{1/T}) \geq -\frac15 \lambda_2.
$$
%%%%%%$$\frac{A(e^{u+v})}{e^{u+v}}=\frac{A(e^u)}{e^u} +\frac{A(e^{u+v})}{e^{u+v}}-\frac{A(e^u)}{e^u} \geq\lambda_2 -\overline{\nu}\left(e^{1/T}; A(x)/x\right)\geq \frac45\lambda_2 >0. $$
Therefore, for any $u\in\RR$ with $g_\si(u) >\lambda_2$, and for any $0\leq v\leq 10/T, 0<\si\leq \si_0=\min(1,\frac{\log 2}{10} T)$, we have
\begin{align*}
\frac{g_\si(u+v)}{g_\si(u)}& =\frac{A(e^{u+v})/e^{u+v}}{A(e^u)/e^u} ~\cdot~ \frac{e^{-\si(u+v)}(1-e^{-\si(u+v)})}{e^{-\si u}(1-e^{-\si u})} \\
& \geq  \left( 1- \frac{\lambda_2/5}{A(e^u)/e^u}\right) \cdot e^{-\si v} %%%\frac{1-e^{-\si(u+v)}}{1-e^{-\si u}}
\geq \frac45 e^{-10\si_0/T} \geq \frac25.
\end{align*}

To estimate $\|g_\si\|_\infty$, we next use the first part of Lemma \ref{lemf} with $\lambda_1=2/5$, and the above chosen $\lambda_2$ and $\lambda_3$, together with \eqref{gsigmachapeau}. We get
\begin{align}\label{eq:gsigmanorm}
\|g_\sigma\|_\infty &\leq 0.55 \int_{-T}^{T} |\widehat{g_\si}(\tau)| d\tau + \max(\lambda_2,\lambda_3)
\notag \\&
\leq 50\overline{\nu}(e^{1/T})+\|A\|_{[1,e]}+\lambda_3+0.6 \eta(\sigma,T)
\notag \\ &
+6.6\sum_{k=0}^{n}\sum_{l=1}^{m_k}|c_{k,l}| \si^{-\omega_{k,l}}.
\end{align}
In turn, this estimate can now be substituted into \eqref{eq:gdiff2} resulting in
\begin{align*}%%%%%\label{eq:gdifffin}
g_\sigma(u+v)-g_\sigma(u) \geq &  -\nu(e^{1/T})-\Phi(u)- \frac{1}{T} \Bigg\{ 50\overline{\nu}(e^{1/T})+\|A\|_{[1,e]} +2M\si
\\ & +\lambda_3+0.6 \eta(\sigma,T) + 6.6\sum_{k=0}^{n}\sum_{l=1}^{m_k}|c_{k,l}| \si^{-\omega_{k,l}}
\Bigg\}.
\end{align*}
Substituting this into \eqref{eq:Ldifference} leads, for $\sigma\in(0,\si_0)$, $u\in \RR$ and $0<v<1/T$ to$$
L_\sigma(u+v)-L_\sigma(u)
\geq -\Phi(u) - \nu(e^{1/T})-\frac{C_6(\si,T)}{T} -\frac{4}{\sqrt\pi}\si\sum_{k=1}^{n}\sum_{l=1}^{m_k}\frac{|c_{k,l}|}{T^{\omega_{k,l}+1}}~,
$$
with
\begin{align*}
C_6(\sigma,T)&:=50\overline{\nu}(e^{1/T})+\|A\|_{[1,e]}+\lambda_3+ 2M\si + 0.6 \eta(\sigma,T)\\
&+\sum_{k=0}^{n}\sum_{l=1}^{m_k}|c_{k,l}|\left(\frac{4}{\sqrt\pi}(\sigma+|b_k|)+6.6 \right)\si^{-\omega_{k,l}}.
\end{align*}

Finally, we apply Lemma \ref{l:T} to $-L_\sigma$, with $\Phi$ the even function occurring here and the role of $b$ played by the rest of the lower estimate for the difference of $L_\sigma(u+v)-L_\sigma(u)$. Note that then $b> 50 \overline{\nu}(e^{1/T})/T \geq 50 \{\nu(e^{1/T})+\Phi(0)\}/T \geq 50 \Phi(0)/T$, whence $x_0(T,b)\leq (16/T)\max(\sqrt{1.02},\sqrt{1.02T})<\max(17/T,17/\sqrt{T})$.
It follows from Lemma \ref{l:T}, \eqref{eq:Ten} and \eqref{Lsigmachapeau} that for $t>\max(17/T,17/\sqrt{T})$
$$
|L_\si(t)| \leq 20\left(\sqrt{\Phi(0)\Phi(t/2)} + \sqrt{\Phi(0) \left(b+\frac1\pi \eta(\si,T)\right)} + b+\frac1\pi \eta(\si,T)\right),
$$
whenever $0<\si<\min(1,\frac{\log 2}{10} T )$. Choosing in this estimate $t:=1/\si$ and $x:=e^t$, we get for $x \geq \max(e,e^{17/T},e^{17/\sqrt{T}})$
$$
\left\vert A(x)-x\sum_{k=0}^{n}\sum_{l=1}^{m_k}\Re\left(c_{k,l}e^{-ib_k\log{x}}\right) \frac{(\log{x})^{\omega_{k,l}}}{\Gamma(\omega_{k,l}+1)}\right\vert \leq x\rho_T(x)
$$
with
$$
\rho_T(x)=\frac{20 e}{1-e^{-1}}\left(\sqrt{\Psi_{e^{1/T}}(1)\Psi_{e^{1/T}}(\sqrt{x})} +R_T(x)+\sqrt{\Psi_{e^{1/T}}(1)R_T(x)}\right)
$$
where
\begin{align}\label{eq:firstRT}
R_T& (x)= \nu(e^{1/T}) + \left(\frac1\pi +\frac{0.6}{T}\right) \eta\left(\frac{1}{\log x},T\right)
\\& \notag + \left(8\overline{\nu}(e)+8|A(1)|+K\right) %%%%%%%%%% \times
\frac{\log^{\min(1,\Om_{+})} x +\log (2 \log x)}{T}
\\& \notag +\frac{4}{\sqrt\pi}\frac{1}{T\log{x}}\sum_{k=1}^{n}\sum_{l=1}^{m_k}|c_{k,l}| \left(\frac{1}{T^{\omega_{k,l}}}+\log^{\omega_{k,l}} x \right)
\\ & +\frac1T \bigg(50\overline{\nu}(e^{1/T})
+\|A\|_{[1,e]}+\sum_{k=0}^{n}\sum_{l=1}^{m_k}|c_{k,l}| \left( \frac{4}{\sqrt\pi}|b_k|+6.6 \right)\log^{\omega_{k,l}}x \bigg).\notag
\end{align}
That concludes the proof.
\end{proof}

The above result has a point only for $T$ large, in particular when $T$ can be fixed arbitrarily large. For $T\geq 1$ the theorem gets a slightly simpler form as follows.

\begin{corollary}\label{c:kle1}
Let $A(x)$ be a function satisfying the same hypothesis as in Theorem \ref{th:kle1} and let $T\geq 1$. Define as previously the functions $G$ and $\eta(\sigma,T)$ ($0<\sigma<1$).
Then for $x\geq  \max(e,e^{17/\sqrt{T}})$, we have
$$
\left\vert A(x)-x\sum_{k=0}^{n}\sum_{l=1}^{m_k}\frac{(\log{x})^{\omega_{k,l}}}{\Gamma(\omega_{k,l}+1)}
2\Re\left(c_{k,l}e^{-ib_k\log{x}}\right)\right\vert \ll  x\rho_T(x)
$$
with the implicit constants only depending on $\overline{\nu}(e)$, $c_{k,l}$, $\omega_{k,l}$, $b_k$, $\AAA(2)$, $\|A\|_{[1,e]}$ and the implicit constant occurring in \eqref{eq:etacondifin}, and with
$$
\rho_T(x) =\sqrt{\Psi_{e^{1/T}}(1)\Psi_{e^{1/T}}(x)}+R_T(x)+\sqrt{\Psi_{e^{1/T}}(1)R_T(x)}
$$
where with $\omega:=\min_{k,l}\omega_{k,l}$ and $\Om:=\max_{k,l} \omega_{k,l}$
$$
R_T(x)=  \nu(e^{1/T}) + \eta\left(\frac{1}{\log x},T\right) + \frac{\log^{\Om_+} x+\log\log x}{T} + \frac{1}{T^{\omega+1}\log{x}}.
$$
\end{corollary}

Let us give a small discussion of the last result. On the one hand it uses the very condition of "$A(x)/x$ is slowly decreasing", the condition appearing already in the earliest proofs of the Wiener-Ikehara Theorem. On the other hand it does not cover by itself this result, for we did not assume that strong assumptions on the boundary function. That slight shortcoming is remedied by the fact that for $A(x)$ increasing, we have the needed asymptotic evaluation anyway by Theorem \ref{th:kle}.

The classical proof of the Wiener-Ikehara Theorem was worked out in the case when we have $k=0$, $m_0=0$, that is only one term in the singular asymptotic, with $\omega_{0,0}=0$: $\AAA(s+1)\sim c/s$. Then the argument starts with proving that once $A(x)$ is increasing, we as well have $A(x)=O(x)$, a statement fully contained in Corollary \ref{c:meromorphiclocal} if the Laplace transform has a meromorphic continuation at $s=1$. This can easily be proved also under other, less stringent conditions like e.g. a continuous, or an $L^1$ boundary function on $[1-iT,1+iT]$ with some $T>0$.

Here in Theorem \ref{th:kle1} we settled with the more general condition of $\eta(\si,T)=O(1)$ ($\si\to 0$) with some $T>0$, and thus with the slightly less precise bound $A(x)=O(x\log\log x)$, coming from the general (lower) estimation of slowly decreasing functions \eqref{minslowdec}. On the other hand for the special case when $A(x)$ is increasing, we can as well use that $A(x)$ is moderately decreasing, and apply Theorem \ref{th:kle}, which already concludes, from only boundedness of $\eta$, the estimate $A(x)=O(x)$. (Note that in general $A(x)/x$ being slowly decreasing need not entail that $A(x)$ be moderately decreasing, so even the less increasing, as discussed in Section \ref{sec:slowdecrease}.) But once we know $A(x)=O(x)$, we then can conclude from $A(x)$ being increasing, (and also if only $A(x)$ is slowly decreasing, see Corollary \ref{cor:dividebylog}), that also  $A(x)/x$ is slowly decreasing, and thus also Theorem \ref{c:kle1} can be applied. However, this is kind of superfluous once Theorem \ref{th:kle} have already been used anyway: the two give very similar error bounds (with the former being slightly better in view of the occurrence of $\log\log x$ in the latter).

If for the application of these results we have a situation where $\Omega>0$, then the slight loss regarding the occurrence of $\log\log x$ diminishes and the order of $A(x)$ is to be $x\log^{\Om_{+}} x$, anyway. But in this case starting out even from an increasing function $A(x)$ it is not clear, how slow decrease of $A(x)/x$ may follow, if at all, while $A(x)$ is still of moderate decrease and Theorem \ref{th:kle} still works. To apply Theorem \ref{th:kle1}, one may have to consider the order of $\eta$, or other information about the behavior of $A(x)$. Alternatively, at all probability an analogous theorem, under the assumption that $A(x)/x\log^{m} x$ with $m=\Om_{+}$ is slowly decreasing, can similarly be proved, similarly as we have done in case of Theorem \ref{th:kleTfixe}. That would directly work here after an application of Proposition \ref{prop:dividebyfi} with $\ell(x):=x\log^{\Om_{+}}x$.

We spare the reader from these calculations, for in all cases a direct use of Theorem \ref{th:kle} with the increasing, or moderately decreasing properties of $A(x)$ would give a simpler solution. However, let us note that such a result would be slightly more general (as slow decrease of $A(x)/x$ would imply that of $A(x)/x\log^{\Om_{+}} x$, too, and as also the moderate decrease of $A(x)$ implies this condition -- assuming in both cases $A(x)=O(x\log^{\Om_{+}} x)$). Compare Remark \ref{rem:retro}.

\noindent
{\sc\small
IECN- Universit\'e Henri Poincar\'e : Nancy 1 \\ BP 239\\
F-54 506 VANDOEUVRE-L\`ES-NANCY cedex, FRANCE \\
E-mail: {\tt deroton@iecn.u-nancy.fr}\\
\ \\
and\\
\ \\
{\sc\small
Alfr\' ed R\' enyi Institute of Mathematics, \\ Hungarian Academy of Sciences, \\
1364 BUDAPEST, P.O.B. 127, HUNGARY}\\
E-mail: {\tt revesz@renyi.hu}\\
\smallskip
\indent and\\
\smallskip
\noindent
{\sc\small
Department of Mathematics, \\ Kuwait University \\
P.O. Box 5969 Safat -- 13060 KUWAIT}\\
%%%% E-mail: {\tt szilard@sci.kuniv.edu.kw}

\end{document}